\documentclass[reqno,11pt]{amsart}
\usepackage{amsmath, latexsym, amsfonts, amssymb, amsthm, amscd}
\usepackage{graphicx}
\usepackage[utf8]{inputenc}
\usepackage{stmaryrd}
\usepackage{mathabx}
\usepackage{hyperref}
\usepackage{xcolor}
\usepackage{etoolbox}
\usepackage{cleveref}

\usepackage{enumitem}
\setlist{itemsep=1pt,topsep=3pt,parsep=1pt,leftmargin=\parindent,itemindent=\parindent}
\setlist[itemize]{itemsep=1pt,topsep=2pt,parsep=1pt,wide}

\usepackage[left=2.2cm,right=2.2cm,top=2.2cm,bottom=2.2cm]{geometry}
\renewcommand{\bar}{\overline}

\newcommand{\lint}{\llbracket}
\newcommand{\rint}{\rrbracket}

\numberwithin{equation}{section}

\newtheorem{theorem}{Theorem}[section]

\newtheorem{lemma}[theorem]{Lemma}
\newtheorem{proposition}[theorem]{Proposition}

\newtheorem{rem}[theorem]{Remark}

\newtheorem{claim}[theorem]{Claim}

\newcommand{\dd}{\mathrm{d}}
\newcommand{\ind}{\mathbf{1}}

\renewcommand{\P}{\mathrm{P}}
\newcommand{\E}{\mathrm{E}}
\renewcommand{\tilde}{\widetilde}
\renewcommand{\hat}{\widehat}
\newcommand{\cc}{\complement}
\newcommand{\bt}{\mathbf{t}}
\newcommand{\bbVar}{\mathbb{V}\mathrm{ar}}

\newcommand{\cX}{{\ensuremath{\mathcal X}} }
\newcommand{\cA}{{\ensuremath{\mathcal A}} }

\newcommand{\cF}{{\ensuremath{\mathcal F}} }
\newcommand{\cJ}{{\ensuremath{\mathcal J}} }
\newcommand{\cP}{{\ensuremath{\mathcal P}} }

\newcommand{\cN}{{\ensuremath{\mathcal N}} }
\newcommand{\cL}{{\ensuremath{\mathcal L}} }

\newcommand{\bP}{{\ensuremath{\mathbf P}} }
\newcommand{\bQ}{{\ensuremath{\mathbf Q}} }
\newcommand{\bE}{{\ensuremath{\mathbf E}} }

\newcommand{\cW}{{\ensuremath{\mathcal W}} }
\newcommand{\cU}{{\ensuremath{\mathcal U}} }

\DeclareMathSymbol{\leqslant}{\mathalpha}{AMSa}{"36} % nicer `smaller or equal'
\DeclareMathSymbol{\geqslant}{\mathalpha}{AMSa}{"3E} % nicer `larger or equal'
\DeclareMathSymbol{\eset}{\mathalpha}{AMSb}{"3F}     % nicer `emptyset'
\renewcommand{\leq}{\;\leqslant\;}                   % redef. of < or \renewcommand{\geq}{\;\geqslant\;}                   % redef. of > or \newcommand{\dd}{\,\text{\rm d}}             % a straight d for differentials

\newcommand{\Var}{\mathrm{Var}}        % \sum-like symbol for union
\newcommand{\prodtwo}[2]{\prod_{\substack{#1 \\ #2}}}     % product 2 lines
\newcommand{\bbE}{{\ensuremath{\mathbb E}} }

\newcommand{\bbP}{{\ensuremath{\mathbb P}} }

\newcommand{\bbR}{{\ensuremath{\mathbb R}} }

\newcommand{\bbZ}{{\ensuremath{\mathbb Z}} }

\newcommand{\gb}{\beta}
\newcommand{\gep}{\varepsilon}       % \ge already exists...
\newcommand{\gp}{\varphi}

\newcommand{\gl}{\lambda}

\makeatletter
\def\captionfont@{\footnotesize}
\def\captionheadfont@{\scshape}

\long\def\@makecaption#1#2{%
  \vspace{2mm}
  \setbox\@tempboxa\vbox{\color@setgroup
    \advance\hsize-6pc\noindent
    \captionfont@\captionheadfont@#1\@xp\@ifnotempty\@xp
        {\@cdr#2\@nil}{.\captionfont@\upshape\enspace#2}%
    \unskip\kern-6pc\par
    \global\setbox\@ne\lastbox\color@endgroup}%
  \ifhbox\@ne % the normal case
    \setbox\@ne\hbox{\unhbox\@ne\unskip\unskip\unpenalty\unkern}%
  \fi
  \ifdim\wd\@tempboxa=\z@ % this means caption will fit on one line
    \setbox\@ne\hbox to\columnwidth{\hss\kern-6pc\box\@ne\hss}%
  \else % tempboxa contained more than one line
    \setbox\@ne\vbox{\unvbox\@tempboxa\parskip\z@skip
        \noindent\unhbox\@ne\advance\hsize-6pc\par}%
\fi
  \ifnum\@tempcnta<64 % if the float IS a figure...
    \addvspace\abovecaptionskip
    \moveright 3pc\box\@ne
  \else % if the float IS NOT a figure...
    \moveright 3pc\box\@ne
    \nobreak
    \vskip\belowcaptionskip
  \fi
\relax
}
\makeatother
\def\writefig#1 #2 #3 {\rlap{\kern #1 truecm
\raise #2 truecm \hbox{#3}}}

\newcommand{\tf}{\mathtt{F}}

\newcommand{\cons}{\texttt{c}}

\title[Random Walk Pinning Model II: upper bounds and disorder relevance]{The Random Walk Pinning Model  II:\\ 
Upper bounds on the free energy and disorder relevance}
\author{Quentin Berger}
\address{Université Sorbonne Paris Nord, Laboratoire d'Analyse, Géométrie et Applications, CNRS UMR 7539, 99 Av. J-B Clément, 93430 Villetaneuse, France and Institut Universitaire de France}
\email{quentin.berger@math.univ-paris13.fr}
\author{Hubert Lacoin}
\address{IMPA, Estrada Dona Castorina, 110,
Rio de Janeiro, Brazil}
\email{lacoin@impa.br}

\date{\today}

\begin{document}

\begin{abstract}
This article investigates the question of disorder relevance for the continuous-time Random Walk Pinning Model (RWPM) and completes the results of the companion paper~\cite{BLirrel}.
The RWPM considers a continuous time random walk $X=(X_t)_{t\geq 0}$, whose law is modified by a Gibbs weight given by $\exp(\beta \int_0^T \mathbf{1}_{\{X_t=Y_t\}} \mathrm{d}t)$, where $Y=(Y_t)_{t\geq 0}$ is a quenched trajectory of a second (independent) random walk and $\beta \geq 0$ is the inverse temperature, tuning the strength of the interaction.
The random walk~$Y$ is referred to as the \textit{disorder}. 
It has the same distribution as~$X$ but a jump rate $\rho \geq 0$, interpreted as the \textit{disorder intensity}.
For fixed $\rho\ge 0$, the RWPM undergoes a localization phase transition as $\beta$ crosses a critical threshold $\beta_c(\rho)$.
The question of disorder relevance then consists in determining whether a disorder of arbitrarily small intensity~$\rho$ changes the properties of the phase transition.
We focus our analysis on the case of transient \textit{$\gamma$-stable walks} on $\mathbb{Z}$, \textit{i.e.}\ random walks in the domain of attraction of a $\gamma$-stable law, with $\gamma\in (0,1)$. 
In the present paper, we show that disorder is relevant when $\gamma \in (0,\frac23]$, namely that $\beta_c(\rho)>\beta_c(0)$ for every $\rho>0$.
We also provide lower bounds on the critical point shift, which are matching the upper bounds obtained in~\cite{BLirrel}.
Interestingly, in the \textit{marginal case $\gamma = \frac23$}, disorder is \textit{always relevant}, independently of the fine properties of the random walk distribution; this contrasts with what happens in the marginal case for the usual disordered pinning model.
When $\gamma \in (\frac23,1)$, our companion paper~\cite{BLirrel} proves that disorder is irrelevant (in particular $\beta_c(\rho)=\beta_c(0)$ for $\rho$ small enough).
We complete here the picture by providing an upper bound on the free energy in the regime $\gamma\in (\frac 2 3,1)$ that highlights the fact that although disorder is irrelevant, it still has a non-trivial effect on the phase transition, at any $\rho>0$.
\end{abstract}

\maketitle

% \setcounter{tocdepth}{1}
% \tableofcontents

\section{Introduction and main results}

We consider in this article and its companion paper~\cite{BLirrel} the question of disorder relevance for the Random Walk Pinning Model (RWPM), studied in \cite{BT10,BL11,BS10,BS11}.
In this introduction, we only present the specific technical setup studied in this paper, but we refer to \cite{BLirrel} for a broader overview of the RWPM, together with more complete references.

\subsection{The \texorpdfstring{\(\gamma\)}{gamma}-stable continuous-time RWPM}

We let $J: \bbZ \to \bbR_+$ be a \textit{symmetric} function on $\bbZ$ such that $\sum_{x\in \bbZ^d} J(x) =1$. 
Assume furthermore that $J$ is a \textit{non-increasing} function of $|x|$. 
We then let $W=(W_t)_{t\geq 0}$ be a continuous-time random walk on $\mathbb Z$ with transition kernel~$J(\cdot)$, \textit{i.e.}~$W$ is a continuous time Markov chain with generator $\cL$ given by
\begin{equation*}
\mathcal L f(x)= \sum_{y\in\bbZ^d} J(y)\left(f(x+y)-f(y)\right) \,,
\end{equation*}
and we denote by $\P$ the distribution of $W$. 
We further assume that $W$ is  in the domain of attraction of a $\gamma$-stable process, with $\gamma \in (0,1)$, or more precisely that 
\begin{align}\label{JPP}
  J(x)=\varphi(|x|)(1+|x|)^{-(1+\gamma)} \,,\qquad x\in \bbZ \,,
\end{align}
where $\varphi(\cdot)$ is a slowly varying function, \textit{i.e.}\ such that \(\lim_{x\to\infty}\varphi(cx)/\varphi(x) =1\) for any \(c>0\), see \cite{BGT89}.
Let us note that, since \(\gamma \in (0,1)\), the random walk \(W\) is transient.

Given $\rho \in [0,1)$, we consider  $X,Y$  two \emph{independent} continuous-time random walks with the same transition kernel \(J(\cdot)\) as \(W\), but with respective jump rates $(1-\rho)$ and $\rho$. 
In other words, we can write
\[
X_t = W^{(1)}_{(1-\rho)t} \quad \text{ and }\quad Y_t = W^{(2)}_{\rho t} , 
\]
where $W^{(1)},W^{(2)}$ are two independent copies of $W$.
Since~$X$ and~$Y$ play different roles, we use different letters to denote their distribution: we let $\bP_{1-\rho}$ (or simply~$\bP$) denote the law of~$X$ and $\bbP_{\rho}$ (or simply~$\bbP$) the law of~$Y$. 
% In most instances the superscripts $\rho$ and $(1-\rho)$ will not appear to lighten the notation.
Given $T>0$ (the polymer length) and a fixed realization of $Y$ (\emph{quenched disorder}) we define an energy functional on the set of trajectories by setting
\[
H^Y_T(X):=\int^T_0 \ind_{\{X_t=Y_t\}} \dd t \,.
\]
Then, given $\gb>0$ (the inverse temperature), the Random Walk Pinning Model (RWPM) is defined as the probability distribution $\bP_{\gb,T}^Y$ which is absolutely continuous with respect to $\bP$, with Radon--Nikodym density given by
\begin{equation}
\label{def:gibbs}
\frac{\dd \bP^{Y}_{\gb,T}}{ \dd \bP} (X)  := \frac{1}{Z_{\gb,T}^{Y}}\, e^{\beta H^Y_T(X)} \,, \quad  \text{ where } \quad 
Z_{\gb,T}^{Y} := \bE\Big[ e^{\beta H^Y_T(X)} \Big] \,.
\end{equation}
 The renormalization factor $Z_{\gb,T}^{Y}$ makes $\bP_{\gb,T}^Y$ a probability measure and is referred to as the 
partition function of the model. 
When compared with $\bP$, the measure $\bP^{Y}_{\gb,T}$ favors trajectories $(X_t)_{t\ge 0}$ which overlap with $Y$ within the time interval $[0,T]$.
For convenience a \textit{constrained boundary} analogue of the partition function is defined by adding constraint $X_T=Y_T$ 
and a multiplicative factor $\beta$:
\begin{equation}
 Z_{\gb,T}^{Y,\mathrm{c}} := \beta\bE\Big[ e^{\beta H^Y_T(X)} \ind_{X_T=Y_T}\Big] \,.
\end{equation}

\subsection{Free energy, phase transition and annealing}

We introduce the free energy of the model and the critical point \(\beta_c(\rho)\)  which marks a localization phase transition (we refer to \cite[App.~A]{BLirrel} for a proof). 

\begin{proposition}
\label{freeenergy}
The \emph{quenched free energy}, defined by
\[
\tf(\rho,\beta):= \lim_{T\to \infty} \frac{1}{T} \log  Z^{Y}_{\beta,T} = \lim_{T\to \infty} \frac{1}{T}\bbE\left[ \log  Z^{Y}_{\beta,T}\right],\]
exists for every $\rho\in [0,1)$ and $\beta>0$ and the convergence holds $\bbP$-almost surely and in $L^1(\bbP)$.
 It satisfies the following properties:
(i) for every $\beta$ and $\rho$, $\tf(\rho,\beta)\ge 0$;
(ii) the function $\beta\mapsto \tf(\rho,\beta)$ is non-decreasing and convex;
(iii) the function $\rho\mapsto \tf(\rho,\beta)$ is non-increasing.
We can then define the critical point 
\[
\beta_c(\rho):= \inf\big\{ \gb > 0 \,:\,  \tf(\rho,\gb) >0 \big\} \,,
\]
and we have: (iv) the function $\rho\mapsto \beta_c(\rho)$ is non-decreasing.
\end{proposition}

\noindent We introduce a specific notation for the (constrained) partition function in the specific \textit{homogeneous} case $\rho=0$, setting 
\begin{equation}
  \label{annealedpartz}
 z^{\cons}_{\beta,T}:=\beta \E\left[ e^{\beta \int^T_0  \ind_{\{W_s=0\}} \dd s}\ind_{\{W_T=0\}}\right] \,,
\end{equation}
and we also denote the homogeneous free energy simply by \(\tf(\beta) = \tf(0,\beta)= \lim_{T \to \infty} \frac{1}{T}\log  z^{\cons}_{\beta,T} \).
Let us stress that the homogeneous free energy has an implicit representation: setting
\begin{equation}
\label{def:beta0}
\beta_0:= \left(\int^{\infty}_0 \P(W_t=0) \dd s\right)^{-1} \,,
\end{equation}
we have \(\tf(\beta)=0\) if \(\beta\leq \beta_0\) and \( \int_0^{+\infty} e^{- \tf(\beta) t} \P(W_t=0) \dd t =  \beta^{-1}\) if \(\beta\geq \beta_0\).
Note also that, since \(W\) is transient, we have \(\beta_0>0\).
The computation of the asymptotic properties of $\P(W_t=0)$ (using the local limit theorem, see \cite[Chapter~9]{GK68} or \Cref{sec:prelim} below) coupled with some Tauberian computation allows to deduce the following asymptotic for $\tf(\beta)$ (we refer to \cite[Theorem~2.1]{Gia07} for the analogous result in a discrete time setting and its proof).
\begin{proposition}\label{homener}
The homogeneous free energy has the following critical behavior:
\begin{equation*}
\tf(\gb) \stackrel{\beta\downarrow \beta_0}{\sim}  (\gb-\gb_0)^{\nu} \hat L\left(\frac{1}{\gb-\gb_0}\right) ,
\end{equation*}
 where $\nu=\frac{\gamma}{1-\gamma}\wedge 1$ and $\hat L$ is a (explicit) slowly varying function.  The function $\hat L$ can be replaced by a constant when $\varphi$ is asymptotically constant and $\gamma\ne 1/2$.
\end{proposition}
Note that, for any $\rho\in (0,1)$, we have $\bbE[Z^{Y,\mathrm c}_{\beta,T}]=  z^{\cons}_{\beta,T}$ since \(X-Y \stackrel{(d)}{=}W\).
(This is in fact the main reason why we choose \(X,Y\) to have jump rates \(1-\rho,\rho\) respectively: the annealed model has then no dependence on \(\rho\) anymore.)
Moreover, as a particular case of item $(iii)$-$(iv)$ in \Cref{freeenergy} above, we have 
\begin{equation*}
 \tf(\beta,\rho)\le \tf(\beta) \quad \text{ and } \quad  \beta_c(\rho)\ge \beta_c(0)=\beta_0.
\end{equation*}
In this paper we investigate how accurate the above inequalities are by establishing improved upper bounds on the free energy.

\subsection{Main results}

We divide our results into two parts: the \textit{relevant disorder} regime \(\gamma\in (0,\frac23]\) where we prove a critical point shift \(\beta_c(\rho)>\beta_0\) for any \(\rho>0\), and the \textit{irrelevant disorder} regime \(\gamma \in (\frac23,1)\) where we prove better upper bounds on the free energy (and on the partition function at criticality).

\subsubsection{The relevant disorder regime $\gamma\in (0,\frac{2}{3}]$}

Our first results give lower bounds on the critical point shift in the case when $\gamma\le \frac23$ (which corresponds to $\nu\le 2$).
Let us start with the case $\gamma \in (0,\frac23)$.
For simplicity, to avoid spurious slowly varying factors (that would not have much effect in the proof), we assume in that case that $\varphi$ tends to one; for the same reason, we also exclude the special case $\gamma=\frac12$. 
We therefore only treat the case $\gamma\in (0,\frac23)\setminus\{\frac12\}$ and we suppose that
\begin{equation}
  \label{simpleJPP}
J(x)\stackrel{  |x|\to\infty }{\sim} |x|^{-(1+\gamma)}. 
\end{equation}
\begin{theorem}\label{rele1}
  Assume that~\eqref{simpleJPP} holds with $\gamma\in (0,\frac23)\setminus\{\frac12\}$.
  Then there is some constant $c=c(J)>0$ such that for any \(\rho \in (0,\frac12)\)
  \begin{equation}\label{lkl2}
  \beta_c(\rho)-\beta_0\ge c\, \rho^{\frac{1}{2-\nu}}  \qquad \text{with}\quad  \nu= 1\vee \frac{\gamma}{1-\gamma}\,.
  \end{equation}
  Let us note that \( \frac{1}{2-\nu} = \frac{1-\gamma}{2-3\gamma}\vee 1\).
\end{theorem}

\noindent
Let us mention that \cite[Theorem 2.3]{BLirrel} proves that this lower bound is sharp: it shows that there is a constant \(C>0\) such that \(\beta_c(\rho)-\beta_0\le C \rho^{\frac{1}{2-\nu}}\) for \(\rho \in (0,\frac12)\).

Let us now turn to the marginal case \(\gamma=\frac23\).
In this case, we work with a slowly varying function~\(\varphi\) in \eqref{JPP}, and we show that there is \textit{always} a critical point shift \(\beta_c(\rho)>\beta_0\), no matter what the slowly varying function \(\varphi\) is.
For the ease of the exposition, we only highlight the lower bound on the critical point shift obtained in the case where \(\varphi\) is asymptotic to a power of \(\log\). The expression of the lower bound in the general case, which is more involved, is given in \Cref{prop:key}-\ref{iii2/3} below. 

\begin{theorem}\label{rele2}
Assume that \eqref{JPP} holds with \(\gamma=\frac23\).
Then we have that $\beta_c(\rho)>\beta_0$ for any $\rho >0$.
Furthermore, if \eqref{JPP}  holds with  $\varphi(t) \stackrel{t\to \infty}{\sim}  (\log t)^{\kappa}$ for some $\kappa\in \bbR$, then there exists $c = c(J)>0$ such that, for any $\rho\in (0,\frac12)$ 
\begin{equation}
  \label{soluce2}
  \log (\beta_c(\rho)-\beta_0 ) \ge -c \,
  \begin{cases}  
     \rho^{-\frac{1}{3\kappa}}  & \text{ if } \kappa>1/3 \,,\\ 
     \rho^{-1} \log\big( \tfrac{1}{\rho}\big)   & \text{ if } \kappa=1/3 \,,\\
     \rho^{-1} & \text{ if } \kappa<1/3 \,.
  \end{cases}
 \end{equation}
\end{theorem}

\noindent
Again, let us mention that~\cite[Theorem 2.6]{BLirrel} provides close to matching upper bounds on the critical point shift.
More precisely, for \(\rho \in (0,\frac12)\), \(\log (\beta_c(\rho)-\beta_0 )\) is bounded from above by \(-C \rho^{-1/(1+3\kappa)}\) if \(\kappa>1/3\) and by \(-C \rho^{-1/2}\) if \(\kappa<1/3\) (with a logarithmic correction when \(\kappa=1/3\)).
We believe that the lower bounds of \Cref{rele2} are sharp.

\subsubsection{The irrelevant disorder regime $\gamma\in (\frac23,1)$}

Once again for the sake of making the proof more readable we only consider the case \eqref{simpleJPP}, \textit{i.e.}\ the slowly varyinf function \(\varphi\) tends to one.
We prove in \cite[Theorem 2.1]{BLirrel} that for $\gamma\in (\frac23,1)$ 
\begin{equation}\label{iurl}
  \lim_{\rho\downarrow 0}  \lim_{\beta\downarrow \beta_0}\frac{\tf(\rho,\beta)}{\tf(\beta)} =1 \,,
\end{equation}
which implies in particular that \(\beta_c(\rho)=\beta_0\) for \(\rho\) sufficiently small. 

A natural question is then whether, for some fixed value of $\rho>0$, we have $\tf(\rho,\beta) \stackrel{\beta \downarrow \beta_0}{\sim} \tf(\beta)$. 
The following result yields a negative answer, contrasting with what has been obtained for the disordered pinning model, see~\cite[Thm.~2.3]{GT09}.
\begin{proposition}
  \label{prop:Firrel}
Assume that \eqref{simpleJPP} holds with $\gamma\in (\frac{2}{3},1)$. Then
there exists a constant \(c>0\) such that, for every $\rho\in (0,1)$ we have 
\begin{equation}\label{apenasdif}
 \limsup_{\gb \downarrow \gb_0} \frac{\tf(\rho,\beta)}{\tf(\gb)} \leq 1-c  \rho  <1 \,.
\end{equation}
\end{proposition}
We also prove that the coincidence of the critical point $\beta_c(\rho)=\beta_0$, which holds for small enough~\(\rho\) thanks to~\eqref{iurl}, does not hold all the way up to $\rho=1$.
\begin{proposition}
  \label{prop:largerho}
  Assume that \eqref{simpleJPP} holds with $\gamma\in (\frac{2}{3},1)$.
  Then there exists $\rho_1 \in (0,1)$ such that $\beta_c(\rho)>\beta_0$ for any $\rho \in (\rho_1,1)$.
\end{proposition}
Lastly we show another property to highlight the impact of disorder in the irrelevant regime. 
We show that the normalized point-to-point partition function, at the annealed critical point, goes to~\(0\).
Again, this is in contrast to what happens for the disordered pinning model in the irrelevant disorder regime, see e.g.~\cite{Lac10ecp} (in fact, for the pinning or directed polymer model, the partition function at the annealed critical point vanishes if and only if disorder is relevant).

\begin{proposition}
  \label{prop:partition}
  Assume that \eqref{simpleJPP} holds with $\gamma\in (\frac{2}{3},1)$. 
  Then, there exists a constant \(c>0\) such that, for any \(\rho\in (0,1)\),
  \begin{equation*}
 \lim_{T\to \infty} \bbP\bigg[  \frac{Z^{Y,\cons}_{\beta_0,T}}{z^{\cons}_{\beta_0,T}}\ge T^{-c\rho}\bigg]=0 \,.
  \end{equation*}
\end{proposition}

% \subsection{Comparison with existing results and further comments}
% \label{sec:comments}
% 
% 
% blabla SRW blabla pinning.

\subsection{Comparison with the disordered pinning model}

The results of the present article, combined with~\cite{BLirrel}, give a complete picture regarding the question of disorder relevance for the \(\gamma\)-stable Random Walk Pinning Model.
Let us briefly comment on how our results compare to those obtained for the usual disordered pinning model; we refer to \cite[Section~2.3]{BLirrel} for a more detailed discussion.  

The disordered pinning model is defined as a (discrete-time) renewal process interacting with a defect line with i.i.d.\ pinning potentials, for which the question of disorder relevance has been extensively studied, see~\cite{Gia07,Gia10} for a general overview.
(Let us note that the annealed version of the disordered pinning model coincides with the annealed version of the RWPM.)
In a nutshell, if~\(\nu\) is the critical exponent of the homogeneous free energy (see \Cref{homener}), disorder has been shown to be irrelevant if \(\nu >2\) and relevant if \(\nu <2\).
The results obtained here and in \cite{BLirrel} draw a similar picture for the RWPM. 
 In fact, the critical point shift found when \(\nu>2\) (see \cite[Theorem~2.3]{BLirrel} and \Cref{rele1} above) is of comparable amplitude for both models.
However, there are a couple of important differences between the two models which are highlighted by the results obtained in the present paper.

When \(\nu>2\) (irrelevant disorder regime) the disordered pinning model's free energy displays the same asymptotic behavior at zero as its homogeneous counterpart, see \cite{Ale08,GT09,Ton08a}, and the behavior of the model at criticality is also similar to that of the critical homogeneous model \cite{Lac10ecp}.
For the RWPM however, disorder still has a non-trivial effect both on the free energy curve (cf. \Cref{prop:Firrel}) and on the behavior at criticality (cf. \Cref{prop:partition}).

The difference is even more striking in the marginal case \(\nu=2\).
Indeed, in the \(\nu=2\) case, disorder may be either irrelevant \cite{Ale08,Lac10ecp,Ton08a} or relevant \cite{BL18,GLT10,GLT11} for the disordered pinning model, depending on the fine details of the model, \textit{i.e.}\ on the slowly varying function~\(\varphi\).
As shown in \Cref{rele2}, this is not the case for the RWPM: when \(\nu=2\) disorder is \emph{always relevant}, \textit{i.e.}\ no matter what the slowly varying function~\(\varphi\) is.

The main feature that explains these differences in behavior is the nature of the disorder: in the RWPM, a given jump of the random walk $Y$ has long range effects in the Hamiltonian $H^Y_T(X)$ making \textit{de facto} $(Y_t)_{t\in [0,T]}$ a disorder with a correlated structure (in spite of having independent increments). Besides the differences in behavior noted above, these correlations also make the study of model mathematically more challenging.

\subsection{Some comments on the proof and organisation of the rest of the article}

All of our proofs rely on giving upper bounds on either truncated moments or fractional moments of the partition function.
To obtain these bounds, our first idea is to find  an event \(\cA\) of small probability but which gives an overwhelmingly large contribution to the expectation of $Z_{\beta,T}^{Y}$. We require thus both $\bbP(A)$ and
 \(\tilde \bbP_T(\cA):=\bbE[Z_{\beta,T}^{Y}\ind_{\cA^{\cc}}] / \bbE[Z_{\beta,T}^{Y}]\), called \textit{size-biased probability} of \(\cA^{\cc}\), to be small.
While this is now a standard approach, the main difficulty remains to identify such an event and to prove the desired estimates on the above mentioned probabilities.
From a technical point of view, there are two important ingredients that we use:
\begin{enumerate}
  \item Following \cite{BS10}, we rewrite the partition function as that of a weighted renewal process \(\tau\) (see e.g.~\eqref{goodexpress}), where the weights depends on the increments of \(Y\) on \(\tau\)-intervals.
  This allows us to obtain an intuitive description the size-biased probability, see \Cref{lem:sizebiased}.

  \item  This description of the size biased measure allows in particular to identify one key feature measure. Under $\tilde \bbP_T$, the random walk $Y$ tends to jump less than under the original measure. The mathematically rigorous version of this statement takes the form of a stochastic comparison between the set of jumps under $\bbP$ and $\tilde \bbP_T$ respectively, see \Cref{lem:domination}. 
 The validity of this statement relies on the fact that  \(J(x)\) is non-decreasing in \(|x|\), and its proof
  is based on an unusual Poisson construction of the random walk \(Y\).
\end{enumerate}

Combining these two ingredients we obtain an intuition on the effect of the size biasing on the distribution of $Y$, and this allow us to construct events \(\cA\) suited for the proof of each of the result, and in particular estimate their (size-biased) probability.
While the choice of \(\cA\) depends on the result one wants to prove, it will (most of the time) be based on some statistics counting (large or small) jumps in the Poisson construction, and our task is to understand which range of jump is most affected by the size biasing. Let us underline that \Cref{lem:domination} plays a crucial role in simplifying the computations. 
The stochastic comparison allows us to discard many terms in our first and second moment computations, allowing for a more readable presentation. 
On the other hand let us insist on the fact that this is mostly a practical simplification and plays no role in the heuristic reasoning behind the proof. We believe that our result would still hold without the assumption of monotonicity for $J(\cdot)$, but their proof should require much  heavier computations.

In the context of \Cref{prop:Firrel,prop:partition}, a direct use of the change of measure/size biasing strategy described above is sufficient for our purpose.
On the other hand, in the context of \Cref{prop:largerho}, \Cref{rele1} and Theorem \ref{rele2}, we need to combine it with a (well-established) coarse-graining technique (as in \cite{BT10,BS11}). 
The idea is to break the system into cells whose size is of order $\tf(\beta)^{-1}$ and apply the change of measure/size biasing method to estimate the contribution of each cell to the fractional moment of~$Z^Y_{\beta,T}$.
This allows to take advantage of the \textit{quasi-multiplicative structure} of the fractional moment and state a finite-volume criterion for having \(\tf(\beta,\rho)=0\) (hence \(\beta_c(\rho)\geq \beta\)).
This general framework is identical for the three results, and the choice of event \(\cA\) will differ in all three cases.
Let us stress that for \Cref{rele1}, \(\cA\) will be based on a simple count of jumps of \(Y\).
On the other hand, in the marginal case of \Cref{rele2}, the choice of \(\cA\) is much more involved: it relies on some statistics that counts jumps of \(Y\) with a (very specific) weight that depends on their amplitude, the weight being chosen in such a way that somehow all scales of jumps contribute to the statistics.

\smallskip
Let us now briefly review how the rest of the paper is organized.

\begin{itemize}
  \item In \Cref{sec:prelim}, we present the preliminary properties mentioned above: the rewriting of the partition function, monotonicity properties and Poisson construction of the walk, the interpretation of the size-biased probability (\Cref{lem:sizebiased}) and the stochastic comparison result (\Cref{lem:domination}).
  \item  In Section~\ref{sec:upper}, we prove \Cref{prop:Firrel,prop:partition}, via a simple change of measure argument; it allows in particular to use \Cref{lem:sizebiased} and \Cref{lem:domination} in a simpler context.
  \item In \Cref{sec:method}, we present the general fractional moment/coarse-graining/change of measure procedure, whose goal is to obtain a finite-volume criterion for having \(\tf(\rho,\beta)=0\) for some \(\beta >\beta_0\). 
  This is the common framework for the proofs of \Cref{rele1,rele2} and \Cref{prop:largerho}. 
  \item In \Cref{sec:proofI,sec<2/3,sec:proofII}, we complete the proofs of \Cref{prop:largerho} and \Cref{rele1,rele2} respectively.
  In all cases, we provide the correct change of measure event \(\cA\) and compute all the needed estimates.
\end{itemize}

\section{Preliminary observations and useful tools}
\label{sec:prelim}

\subsection{Rewriting of the partition function}
\label{sec:rewrite}

The first main step is to rewrite the partition function, as done initially in~\cite{BS10} and repeatedly used in the study of the RWPM.
Expanding the exponential $\exp(\int_0^T \ind_{\{X_t=Y_t\}} \dd t)$ appearing in the partition function~\eqref{def:gibbs} and using the Markov property for $X$, we get that
\begin{align*}
Z_{\gb,T}^Y   = 1+ \sum_{k=1}^{\infty} \gb^k \int_{\cX_k(T)} \prod_{i=1}^k\bP\big( X_{t_i-t_{i-1}}  = Y_{t_i} -Y_{t_{i-1}} \big) \dd t_1 \cdots \dd t_k \,,
\end{align*}
where $\cX_k(T):=\{ {\bf t}\in \bbR^k : 0< t_1< t_2<\dots<t_k<T\}$ is the $k$-th dimensional simplex (by convention \(t_0=0\)). 
Noticing that $\bbP\otimes \bP \big( X_{t_i-t_{i-1}}  = Y_{t_i} -Y_{t_{i-1}} \big) = \P(W_t=0)$, we renormalize this function by its total mass (recall the definition~\eqref{def:beta0} of \(\beta_0\)) by setting 
\begin{equation}
  \begin{split}
  \label{def:KKw}
  K_w(s,t,Y)& :=\beta_0\bP\left(X_{t-s}=Y_t-Y_s \right)  \,, \\ 
  K(t)&:=\bbE\left[ K(0,t,Y)\right]=\beta_0 \P(W_t=0) \,.
  \end{split}
\end{equation}
In particular, \(K(t)\) verifies \(\int_0^{\infty} K(t)\dd t =1\).
Plugged in the above expansion for \(Z^Y_{\beta,T}\) and using the same type of expansion for \(Z^{Y,\cons}_{\beta,T}\), we obtain (setting by convention $t_0=0$ and $t_{k+1} =T$)
\begin{equation}
\label{goodexpress}
  \begin{split}
  Z^{Y}_{\beta,T}&=  1+\sum^{\infty}_{k=1} \left(\beta/\beta_0 \right)^{k} \int_{\cX_k(T)} \prod_{i=1}^{k}  K_w(t_{i-1},t_i,Y)  \dd t_j \,,\\
  Z^{Y,\cons}_{\beta,T}&=  \frac{\beta}{\beta_0} K_w(0,T,Y)+\sum^{\infty}_{k=1} \left(\beta/\beta_0 \right)^{k+1} \int_{\cX_k(T)} \prod_{i=1}^{k+1}  K_w(t_{i-1},t_i,Y)  \prod_{j=1}^{k} \dd t_j \,.
  \end{split}
\end{equation}
For the homogeneous model, analogously (or simply using that \(z_{\gb,T}^{\cons} = \bbE[Z^{Y,\cons}_{\beta,T}]\)) we have
\begin{equation}
  \label{goodexpress-hom}
   z_{\gb,T}^{\cons} = \frac{\beta}{\beta_0}K(T) + \sum_{k=1}^{\infty}  (\gb/\gb_0)^{k+1} \int_{\cX_k(T)} \prod_{i=1}^{k+1} K(t_i-t_{i-1}) \prod_{i=1}^k \dd t_i \,.
\end{equation}

\subsection{A continuous-time renewal process and associated pinning model}
\label{sec:annealed}

Consider \(\tau\) a continuous time renewal process with inter-arrival distribution with density \(K(t)\), \textit{i.e.}\ $\tau_0=0$ and $(\tau_{i}-\tau_{i-1})_{i\geq 1}$ are i.i.d.\ with density $K$.
We denote its law by \(\bQ\).
We let $u(\cdot)$ be the renewal density, defined on \((0,\infty)\) by $\int_A u(t)\dd t:= \bQ(|A\cap \tau|)$.
Then, the renewal equation yields
\[
u(T)=K(T) + \sum_{k=1}^{\infty}  \int_{\cX_k(T)} \prod_{i=1}^{k+1} K(t_i-t_{i-1}) \prod_{i=1}^k \dd t_i =z^{\cons}_{\beta_0,t} \,.
\]
Note that for \(\beta\neq \beta_0\), we can also interpret \(z_{\beta,T}^{\cons}\) in terms of a partition function of a pinning model based on the renewal process \(\tau\): from~\eqref{goodexpress-hom}, we have
\begin{equation}
  \label{eq:zpinning}
 z_{\beta,T}^{\cons} = u(T) \bQ\Big[ ( \gb/\gb_0)^{\cN_T} \ \Big| \ T\in \tau\Big] \,,
\end{equation}
where  \(\bQ( \cdot  \ | T \in \tau)=\lim_{\gep \to 0} \bQ( \cdot \mid \tau \cap [T,t+\gep]\neq \emptyset)\) and~\(\cN_T = \max\{k, \tau_k\leq T\} = |\tau\cap [0,T]|\). 
Then, an easy consequence of \cite[Lemma 3.1]{BLirrel} is the following, for any \(A>0\), there exists a constant \(C=C_A\) such that, for any \(\beta \in [\beta_0,2\beta_0]\)
\begin{equation}
  \label{eq:homogeneous}
\forall\, T\leq \frac{A}{\tf(\beta)}\,,  \qquad  C_A^{-1} u(T) \leq  z_{\beta,T}^{\cons} \leq C_A u(T) \,.
\end{equation}
An important point is that our assumption \eqref{JPP} implies that $K(t)$ and $u(t)$ are also regularly varying when $t\to \infty$.
Indeed, recalling that \(K(t) = \beta_0 \P(W_t=0)\), the local limit theorem (see e.g.\ \cite[Ch.~9]{GK68}) implies that $K(t)$ verifies the following asymptotic relation
\begin{equation}
  \label{implicit}
  \varphi\big( 1/K(t)\big) K(t)^{\gamma} \stackrel{t\to \infty}{\sim} \frac{c_{\gamma}}{t}  \,,
\end{equation}
for some explicit constant \(c_{\gamma}>0\) (we also refer to \cite[App.~C]{BLirrel} for details).
In particular, we deduce that there exists a slowly varying function \(L(\cdot)\) such that \(K(t)\) is of the form
\begin{equation}\label{formkt}
 K(t)=  L(t)t^{-(1+\alpha)}  \qquad \text{ with } \alpha = \frac{1-\gamma}{\gamma} \in (0,+\infty) \,.
\end{equation}
We also have
\(
\overline{K}(t) :=\int_t^{\infty} K(s) \dd s \stackrel{t\to \infty}{\sim}\alpha^{-1} L(t)t^{-\alpha}.
\)
Note also that the slowly varying function~\(L(\cdot)\) is asymptotically constant in the case where \(\varphi(\cdot)\) is asymptotically constant.
Concerning $u(t)$, when $\alpha>1$, the continuous-time renewal theorem yields 
\begin{equation}\label{renewalr}
\lim_{t\to \infty} u(t)= \left(\int^\infty_0 sK(s) \dd t\right)^{-1}.
\end{equation}
When $\alpha\in (0,1)$ \cite[Lem.~A.1]{BS11} (see also Topchii~\cite[Thm.~8.3]{Top10}) shows a continuous-time version of Doney's local limit theorem for renewal processes with infinite mean~\cite{D97}: we then have
\begin{equation}
\label{renewal}
u(t) \stackrel{t\to \infty}{\sim} \frac{\alpha \sin(\pi\alpha)}{\pi} \frac{1}{t^2 K(t)} = \frac{\alpha\sin(\pi\alpha)}{\pi}  \frac{\ t^{\alpha-1}}{L(t)}.
\end{equation}

\subsection{Some important properties of the random walk}
\label{sec:propertiesRW}

Let us now present two properties that will be used repeatedly in the article, that both rely on the fact that the function \(J(\cdot)\) is non-increasing in \(|x|\).
The first one is a unimodality and stochastic monotonicity property and the second one is some unusual Poisson construction of the walk which will allow us to compare its law with a size-biased version of it (introduced in Section~\ref{sec:sizebiased} below).

\subsubsection{Unimodality and stochastic monotonicity}

A positive finite measure $\mu$ on $\bbZ$ is said to be \textit{unimodal} if for all $x\in \lint a,b\rint:=[a,b]\cap \bbZ$ we have 
\begin{equation*}
  \mu(x)\ge \min (\mu(a),\mu(b)) \,,
\end{equation*}
where we write $\mu(x)$ for $\mu(\{x\})$ for convenience. 
Additionally, $\mu$ is \textit{symmetric} if $\mu(-x)=\mu(x)$ for every~$x$.   
Obviously positive linear combination of symmetric unimodal measures are symmetric unimodal.
In the paper we make use of the following statement (see e.g. \cite[Problem 26, pp.169]{FellerII} for a continuous version and its proof).

\begin{lemma}
  \label{stablemma}
  The convolution of two symmetric unimodal measures is symmetric unimodal.
\end{lemma}
We  use unimodality as a tool for comparison arguments.
Given $\mu_1$ and $\mu_2$ two symmetric measures we say that $\mu_2$ \textit{stochastically dominates} $\mu_1$, and we write \(\mu_1 \preccurlyeq \mu_2 \), if
\[
\forall k \geq 0\qquad  \mu_1(\lint-k,k\rint)\le \mu_2(\lint-k,k\rint) \,.
\]
When $\mu_1$ and $\mu_2$ are probability measures, this is equivalent to the existence of a coupling of $\xi_1\sim \mu_1$ and $\xi_2\sim \mu_2$ such that $|\xi_1|\le |\xi_2|$ almost surely. 
The following lemma, which is an easy exercise, states that convoluting a symmetric unimodal measure with a symmetric probability stochastically increases the measure.

\begin{lemma}
  \label{monotolemma}
  If \(\mu\) is a symmetric unimodal measure and $f$ a symmetric probability then
  \[
  \mu \preccurlyeq f \ast \mu \,.
  \]
\end{lemma}
Now let us give a couple of consequences for the random walk $(W_t)_{t\geq 0}$.
Our assumption stipulates that \(J(\cdot)\) is a symmetric and unimodal probability, so we obtain that the distribution of $W_t$ (which is a convex combination of $J^{\ast k}$) is symmetric and unimodal as well, for any $t\geq 0$.
\Cref{monotolemma} further implies that the distribution of $W_t$ is stochastically monotone in $t$.
We collect this in the following lemma.

\begin{lemma}
  \label{lem:unimod}
  For all $t>0$, we have
  \begin{equation*}
  \label{eq:unimodalX}
  |x|\le |y| \quad \Rightarrow \quad   \P(W_t =x) \ge \P(W_t=y) \,.
  \end{equation*}
  Additionally, the law of \(W_t\) is stochastically non-increasing in \(t\), in particular, \(t\mapsto \P(W_t=0)\) is non-increasing.
\end{lemma}

\subsubsection{An unusual Poisson construction of the random walk}

The usual construction for a continuous-time random walk with jump rate \(\rho\) consists in adding jumps distributed according to \(J(\cdot)\) at times of a Poisson point process of intensity \(\rho\).
We present instead a different construction that contains extra information in the Poisson Point Process that we use to derive stochastic comparisons.
Let us define a finite measure on $\bbZ_+\times \bbZ$ by setting
\[
\mu(k,x):= (J(k)-J(k+1))\ind_{\{|x|\le k\}} ,\quad \text{ for } k \in \bbZ_+ , x \in \bbZ \,.
\]
Note that the second marginal of $\mu$ corresponds to $J(\cdot)$.
Its first marginal is given by 
\[
\bar \mu(k):=(2k+1) (J(k)-J(k+1)) \,.
\]
We consider $\mathcal U$ a Poisson process on $\bbZ_+\times \bbZ \times \bbR$ with intensity given by $\rho\, \mu \otimes  \dd t$, where $\dd t$ is the Lebesgue measure.
We let $(U_i,V_i,\vartheta_i)_{i\ge 1}$ be the sequence of points in $\mathcal U$ ordered by increasing time \((\vartheta_i)_{i\geq 1}\), and we set
\begin{equation}
  \label{def:Y}
 Y_t:=\sum_{i\ge 1} V_i \ind_{\{\vartheta_i\in[0, t]\}} \,.
\end{equation}
This  is indeed a random walk with transition kernel \(J(\cdot)\) (the second marginal of \(\mu\)) and jump rate \(\rho\).
Let us stress that, contrary to the $V_i$'s, the $U_i$'s are not measurable with respect to $(Y_t)_{t\ge 0}$.
Note also that, by construction, conditionally on $(U_i,\vartheta_ i)_{i\ge 1}$, the $V_i$'s are independent and uniformly distributed on $\lint -U_i,U_i\rint$. 
For this reason, for any fixed $t$, the conditional distribution of $Y_t$ given $(U_i,\vartheta_ i)_{i\ge 1}$ is a convolution of symmetric unimodal distributions. 
This fact turns out to be really helpful in stochastic comparison and for this reason we use the variables~$U_i$ rather than $V_i$ in our computations, for instance when we want to use a variable that play the role of a ``jump amplitude''. 
We let $\bar \cU$ denote the Poisson process $\bbZ_+ \times \bbR$ obtained when deleting the second coordinate.
In the remainder of the paper, $\bbP$ denotes the probability associated with $\cU$ and $Y$ is defined by \eqref{def:Y}.
Given a set $I\subset \bbR$ we denote by~$\cF_I$ the $\sigma$-algebra generated by $\mathcal U$ with time coordinate in $I$, 
\begin{equation}
  \label{defft}
  \mathcal F_I:=\sigma\left( \mathcal U\cap (\bbZ_+\times \bbZ) \times I \right) \quad \text{ and } \quad \mathcal F_t:= \mathcal F_{[0,t]} \,.
\end{equation}

%
% Alternatively, conditionally on \((V_i,\vartheta_i)_{i\geq 1}\), the $U_i$'s are independent with \(\bbP(U_i \geq k) = \frac{\bar J(k)}{\bar J(|V_i|)}\) for \(k\geq |V_i|\).

\subsection{A weighted measure and a comparison result}
\label{sec:sizebiased}

Let us define
\begin{equation}
  \label{def:w}
  w(s,t,Y):= \frac{K_w(s,t,Y)}{ K(t-s)} = \frac{\bP(X_{t-s} = Y_t-Y_s)}{\P(W_{t-s}=0)} \,,
\end{equation}
and note that this is a non-negative random variable with \(\bbE[w(s,t,Y)] =1\).
In particular, \(w(s,t,Y)\) can be interpreted as probability density with respect to \(\bbP\).
Given a finite increasing sequence $\mathbf{t} = (t_i)_{i\in \lint 0, m\rint}$, let us thus define the following \textit{weighted} measure w.r.t.\ \(\bbP\):
\begin{equation}
  \label{def:Pt}
  \dd  \bbP_{\bt} = \prod_{i=1}^m w(t_{i-1},t_i,Y) \dd \bbP \,.
\end{equation}
Recall that \(\bbP\) is the law of the Poisson point process \(\cU\) so \(\bbP_{\bt}\) is a new law of \(\cU\).
However, we have the nice following description for the probability $\bbP_{\bt}$ in term of how the law of~\(Y\) is modified.
For a process \((A_t)_{t\geq 0}\), we use the notation $A_{[r,s]} = (A_u-A_r)_{u\in [r,s]}$.

\begin{lemma}
  \label{lem:sizebiased}
  For any fixed $\bt = (t_i)_{0\leq i\leq m}$, the following properties hold under $\bbP_{\bt}$:
\begin{enumerate}
  \item \label{i:independence}
  The blocks $(Y_{[t_{i-1},t_i]})_{1\leq i \leq m}$ are independent.
  \item \label{ii:description}
  The distribution of $Y_{[t_{i-1},t_i]}$ is described as follows: for any non-negative measurable \(f\),
  \begin{equation*}
    \bbE_{\bt}\big[ f(Y_{[t_{i-1},t_i]}) \big] =  \E \big[ f(W_{[0,\rho(t_i-t_{i-1})]}) \mid W_{t_{i}-t_{i-1}} =0 \big] \,.
  \end{equation*}
\end{enumerate}
\end{lemma}

\begin{proof}
The first part is obvious from the product structure of $\bbP_{\bt}$.
For the second part, using the definition~\eqref{def:w} of \(w(t_{i-1},t_i,Y)\), we simply write
\[
\bbE_{\bt}\big[ f(Y_{[t_{i-1},t_i]}) \big] =\bbE\big[ f(Y_{[t_{i-1},t_i]}) w(t_{i-1},t_i,Y)\big]
  = \frac{\bbE\otimes \bE\big[ f(Y_{[t_{i-1},t_i]}) \ind_{\{Y_{t_i-t_{i-1}} = X_{t_{i}-t_{i-1}}\}} \big]}{\P(W_{t_i-t_{i-1}}=0)}  \,.
\]
The conclusion follows, recalling that $Y$ and $X$ have jump rates $\rho$ and $1-\rho$ respectively (and \(W\) has jump rate \(1\)).
\end{proof}

We can also compare the weighted measure \(\bbP_{\bt}\) with the original one \(\bbP\), by using the Poisson construction of the previous section.
We equip $\cP\left( \bbZ_+\times \bbR\right)$ with the inclusion order, and we say that a function $\varphi: \cP\left( \bbZ_+\times \bbR\right)\to \bbR_+$ is increasing if $\varphi(U)\le \varphi(V)$ whenever $U\subset V$. 
Recall that $\bar \cU$ denote the Poisson process obtained when ignoring the second coordinate in $\cU$.

\begin{proposition}
\label{lem:domination}
For any non-decreasing function $\varphi: \cP\left( \bbZ_+\times \bbR\right)\to \bbR_+$, we have
\[
\bbE_{{\bf t}}[\varphi(\bar \cU)]\le \bbE[\varphi(\bar \cU)] \,.
\]
\end{proposition}

\begin{rem}
  Let us stress that the analogue result is false if one considers either the full Poisson process~$\cU$ or the Poisson process \((V_i,\vartheta_i)_{i\geq 1}\) usually used to define \(Y\). 
  Indeed, in view of \Cref{lem:sizebiased}-\ref{ii:description} above, because of the conditioning to a future return to \(0\), the presence of a large positive jump for $Y$ makes a large negative jump more likely under \(\bbP_{\bt}\).
\end{rem}

\begin{proof}
By definition of \(\bbP_{\bt}\), we have 
\begin{equation*}
 \bbE_{\bf t}\left[\varphi(\bar \cU)\right]= \bbE\bigg[ \varphi(\bar \cU) \bbE\Big[ \prod_{i=1}^m w(t_{i-1},t_i,Y) \ \Big| \ \bar \cU\Big] \bigg]  \,.
\end{equation*}
Is is enough to show that the conditional expectation
$\bbE\big[ \prod_{i=1}^m w(t_{i-1},t_i,Y) \ | \ \bar \cU\big] $ is a non-increasing function of $\bar \cU$.
Indeed, applying the Harris-FKG inequality (and recalling that \(\bbE[w(s,t,Y)]=1\)) then directly yields the result.
Now, because of the product structure of the measure, it is sufficient to check this for \(m=1\), or more simply put and recalling the definition of \(w(0,t,Y)\), that $\bbE[ \bP(X_t=Y_t) \ | \ \bar \cU]$ is a non-increasing function of $\bar \cU$.

To see this, remark that conditionally on $\bar \cU$, denoting by $\cJ_{t} = |\bar \cU \cap (\bbZ_+\cap [0,t])|$ the number of jumps in the interval \([0,t]\), the distribution of $Y_t$ is given by a convolution of \(\cJ_{t}\) independent random variables \((V_i)_{1\leq i \leq \cJ_t}\) which are uniformly distributed on $\lint -U_i,U_i\rint$ (thus the \(V_i\)'s are symmetric unimodal). 
From \Cref{monotolemma}, each convolution stochastically increases the distribution of $|Y_t|$ which implies that for any non-increasing function \(f:\bbZ_+ \to \bbR_+\) the conditional expectation $\bbE\left[ f(|Y_t|) \ | \ \bar \cU\right]$ is a non-increasing function of $\bar \cU$.
Applying this to the function $y\mapsto \bP(|X_t|=y)$, which is non-increasing by \Cref{lem:unimod}, completes the proof. 
\end{proof}

\section{Proof of \texorpdfstring{\Cref{prop:Firrel,prop:partition}}{two propositions}}
\label{sec:upper}

In this section, we prove~\Cref{prop:Firrel} and \Cref{prop:partition}.
The strategy of the proof consists in estimating a truncated moment of a (modified) partition function, using the perspective of the size biased measure. 
A similar idea is used for the proofs of \Cref{prop:largerho} and \Cref{rele1,rele2}, but in that case a coarse-graining argument is needed, which makes the method more technical (see \Cref{sec:method}).

\subsection{Some notation and preliminaries}

Before we go into the proofs of \Cref{prop:Firrel} and \Cref{prop:partition}, let us introduce some notation (we refer to \cite[Section 3]{BLirrel} for more background).
For \(\gb\geq \gb_0\), define the probability density
\begin{equation}
  \label{def:Kbeta}
K_{\gb}(t) := \frac{\gb}{\gb_0} e^{-\tf(\gb)} K(t) \,,
\end{equation}
and we let \(\bQ_{\gb}\) denote the law of a renewal process with inter-arrival distribution \(K_{\gb}\), note that \(\bQ=\bQ_{\gb_0}\).
Then, in analogy with~\eqref{eq:zpinning}, recalling the definition~\eqref{def:w} of \(w(s,t,Y)\), we can write that 
\begin{equation}
\label{eq:martingale}
 \frac{Z^{Y,\mathrm c}_{\beta,T}}{\bbE[Z^{Y,\mathrm c}_{\beta,T}]} =: \cW^Y_{\beta,T}= \bQ_{\beta,T}\bigg[  \prod_{i=1}^{\mathcal N_T} w(\tau_{i-1},\tau_i,Y)  \bigg] \,,
\end{equation}
where $\bQ_{\beta,T}:=\bQ_{\gb}\left( \cdot \ | \ T\in \tau \right) = \lim_{\gep\downarrow 0} \bQ_{\gb}\left( \cdot \ | \ \tau\cap [T,t+\gep) \neq \emptyset \right)$. (See also Equation~(3.7) in \cite{BLirrel}.)

\subsubsection{Reformulating the results in terms of the normalized partition function}

Both \Cref{prop:Firrel} and \Cref{prop:partition} will follow from estimates on a truncated moment of \(\cW_{\gb,T}^{Y}\).
For instance, \Cref{prop:Firrel} is a consequence of the following.

\begin{proposition}
  \label{prop:apenasdif}
  There is a constant \(c>0\) such that, for any $\rho \in (0,1)$ and any $\beta \in (\beta_0,2\beta_0)$ we have
  \begin{equation}\label{werh}
  \liminf_{T\to\infty} \frac1T \log \bbE\big[ 1\wedge \cW_{\beta,T}^Y  \big] \leq - c \rho \tf(\beta) \,.
  \end{equation}
\end{proposition}

\begin{proof}[Proof of \Cref{prop:Firrel}]
Writing that \(W= (1\wedge W) (1\vee W)\), we get using Jensen's inequality that
\[
\tf(\rho,\gb) -\tf(\gb) = \lim_{T\to\infty} \frac{1}{T} \bbE \log \cW_{\gb,T}^Y \leq  \liminf_{T\to\infty} \frac{1}{T}\log \bbE\big[ 1\wedge \cW_{\beta,T}^Y  \big] + \liminf_{T\to\infty} \frac{1}{T}\log \bbE\big[ 1\vee \cW_{\beta,T}^Y  \big] \,.
\]
As we have \(\bbE[ 1\vee \cW_{\beta,T}^Y ]\leq 1+ \bbE[ \cW_{\beta,T}^Y] =2\), \Cref{prop:apenasdif} shows that \(\tf(\rho,\gb) \leq (1-2\delta) \tf(\gb)\), uniformly in \(\gb \in (\gb_0,\gb_0+1)\).
\end{proof}
\noindent Similarly, \Cref{prop:partition} is a consequence of the following, thanks to a simple application of Markov's inequality.
\begin{proposition}
  \label{prop:martingale}
  Assume that \(J(x)\stackrel{|x|\to \infty}{\sim} |x|^{-(1+\gamma)}\) for some \(\gamma \in (0,1)\).
  Then, there is some constant \(c>0\) such that, for any \(\rho \in (0,1)\) we have for all \(T\) large enough
  \[
  \bbE\left[ 1\wedge \cW_{\beta_0,T}^Y\right] \leq  T^{-c \rho } \,.
  \]
\end{proposition}

\subsubsection{The size-biased perspective}

% In this section we do not need a coarse-graining procedure as in Section~\ref{sec:method} and 
We estimate directly a truncated moment of \(\cW_{\gb,T}^Y\) using the size biased measure.
We use the following:
\begin{lemma}
  \label{lem:StefHub}
  For any event $\cA = \cA_T\in \cF_T$, we have  
  \begin{equation}\label{eriop}
  \bbE\left[ 1\wedge \cW^Y_{\beta,T}\right]\le \bbP(\cA) + \bbE\big[ \cW^Y_{\beta,T} \ind_{\cA^{\cc}} \big] \,.
  \end{equation}
\end{lemma}

\begin{proof}
  We simply use that \(1\wedge \cW^Y_{\beta,T} \leq 1\) on \(\cA\) and \(1\wedge \cW^Y_{\beta,T} \leq 1\) on \(\cA^{\cc}\).
\end{proof}

\noindent
Since \(\cW^Y_{\beta,T}\geq 0\) and \(\bbE[\cW^Y_{\beta,T}]=1\), we can view \(\cW_{\beta,T}^Y\) as a density of a new measure for \(Y\), called \textit{size-biased measure}.
Therefore, our guideline to prove \Cref{prop:apenasdif} or \Cref{prop:martingale} is to find some event \(\cA\) which has small probability under \(\bbP\) but becomes typical under the size biased measure.
This event \(\cA\) will depend on what we need to prove.
Note that, in view of~\eqref{eq:martingale} and recalling the definition~\eqref{def:Pt} of the weighted measure, we have that \(\bbE[ \cW^Y_{\beta,T} \ind_{\cA^{\cc}}] = \bQ_{\gb,T}[ \bbP_{\tau}(\cA^{\cc}) ] \).
To bound this, we will use the following inequality:
introducing an event \(B \in \sigma(\tau\cap[0,T])\), we have
\begin{equation}
  \label{eq:boundtildePAT}
  \bbE\big[ \cW^Y_{\beta,T} \ind_{\cA^{\cc}} \big] = \bQ_{\gb,T}[ \bbP_{\tau}(\cA^{\cc})]  \leq \bQ_{\gb,T}(B^{\cc}) + \bQ_{\gb,T}\big[\bbP_{\tau}(\cA^{\cc}) \ind_{B} \big]  \,.
\end{equation}
Therefore, we need to find some events \(\cA\) and \(B\) such that: \(\bbP(\cA)\) and \(\bQ_{\beta,T}(B^{\cc})\) are small and \(\bbP_{\tau}(\cA^{\cc})\) is small on the event \(B\).

\subsection{Proof of \texorpdfstring{\Cref{prop:apenasdif}}{the proposition}}

Let us first introduce the events \(\cA\) and \(B\) that we use in the proof of \Cref{prop:apenasdif}.
For any \(\gb\in (\gb_0,\gb_0+1)\), let us define for \(0\leq a <b \leq T\)
\[
\cJ_{(a,b]}^{\beta} := \sum_{i: \vartheta_i \in (a,b]} \ind_{\{U_i K(1/\tf(\gb)) \geq \beta_0\}}
\]
the number of ``\(U_i\)-jumps'' larger than \(\beta_0 K(1/\tf(\gb))^{-1}\) in the interval \((a,b]\). The value of the threshold corresponds to the typical maximal amplitude observed in a time interval of length $\tf(\beta)^{-1}$.
Then, given~$\eta$ and $\delta$ two positive parameters (to be fixed later in the proof) we define 
\[
\cA = \cA_T^{\gb} := \big\{ \cJ_{(0,T]}^{\gb} \leq (1-\eta) \bbE[\cJ_{(0,T]}^{\gb}] \big\} \,.
\]
and, letting \(\Delta\tau_j := \tau_j-\tau_{j-1}\),
\[
B = \bigg\{ \sum_{j=1}^{\cN_{T/2}}  \Delta\tau_j\ind_{\{\Delta\tau_j  \tf(\beta)\le 1\}} \ge \delta T \bigg\} \,.
\]
Thanks to \Cref{lem:StefHub} and \eqref{eq:boundtildePAT}, in order to conclude the proof of Proposition \ref{prop:martingale}, we need to prove the following statement: There exists a constant \(c_0>0\) such that, if \(\eta>0\) is small enough and \(\delta\leq 10\eta\), then the following three estimates hold true for all \(T\) sufficiently large
\begin{align}
  \label{eq:logPA}
  \bbP(\cA) &\leq \exp\big(- c_0 \eta^2 \rho \tf(\gb) T\big) \,,\\
  \label{eq:logPtildeA}
  \bbP_{\tau}(\cA^{\cc}) \ind_B &\leq  \exp\big(-c_0 \eta^2 \rho \tf(\gb) T\big) \,, \\
  \label{eq:logPB}
  \bQ_{\beta,T}(B^{\cc}) &\leq \exp\big(- c_0 \tf(\gb) T\big) \leq \exp\big(-  c_0 \rho \tf(\gb) T\big) \,.
\end{align}
Before we start the proof, let us provide some large deviation estimate for a Poisson random variable that we use below. If \(X \sim \mathrm{Poisson}(\lambda)\) for some \(\lambda>0\), then for \(t\in (0,\lambda]\) we have
\begin{equation}
  \label{LDPoisson}
  \bbP\big(  X -\bbE[X] \leq - t \big) , \bbP\big( X -\bbE[X] \geq t \big) \leq \exp\Big( - \frac14 \frac{t^2}{\lambda} \Big)\,.
\end{equation}
The proof is standard and relies on exponential Chernov's inequality.

\begin{proof}[Proof of~\eqref{eq:logPA}]
Under $\bbP$, $\mathcal{J}_{(0,T]}^{\gb}$ is a Poisson random variable of mean \(\bbE[\cJ_{(0,T]}^{\gb}] = \rho f(\gb) T\), where 
\begin{equation}
  \label{def:fb}
  f(\gb) := \sum_{k\geq \beta_0 K(1/\tf(\gb))^{-1}} \bar \mu( k) .
\end{equation}
Using the large deviation~\eqref{LDPoisson} for a Poisson variable, we obtain that \(\bbP(\cA)\le \exp(- \frac{\eta^2}{4} \rho f(\gb) T)\), and it only remains to show that \(f(\beta)\) is of the order of \(\tf(\beta)\).
Recalling that \(\bar\mu(k) = (2k+1) (J(k)-J(k+1)) \), summation by part readily shows that 
\begin{equation}\label{lassum}
\sum_{\ell\geq n} \bar \mu(\ell)  = (2n+1) J(n) + 2\sum_{\ell\geq  n+1} J(\ell) \stackrel{n\to \infty}{\sim} \frac{2(1+\gamma)}{\gamma} nJ(n) \,,
\end{equation}
where we have used regular variation for the last identity.
Therefore, using \eqref{implicit}, we obtain that 
\[
c \tf(\beta) \leq f(\beta)\le c'\tf(\beta)
\]
for some universal constants $c,c'>0$. This concludes the proof.
\end{proof}

\begin{proof}[Proof of \eqref{eq:logPtildeA}] 
We split $\cJ_{(0,T]}^{\beta} = \cJ_1+\cJ_2$ according to the contribution of ``small'' and ``big'' \(\tau\)-intervals respectively (for ):
\[
\cJ_1 := \sum_{j=1}^{\cN_{T/2}}  \cJ_{(\tau_{j-1},\tau_j]}^{\gb} \ind_{\{\Delta\tau_j\tf(\beta) \leq 1\}} \quad \text{ and } \quad \cJ_2 = \cJ_{[0,T]}^{\gb} -\cJ_1 \,.
\] 
The idea is that~\(\cJ_1\) corresponds to the part which is the ``most affected'' by the change of measure from~\(\bbP\) to~\(\bbP_{\tau}\)
We then have that \(\bbP_{\tau}(\cA^{\cc}) \leq \bbP_{\tau}(\cA_1) + \bbP_{\tau}(\cA_2)\), where 
\[
\cA_1 := \big\{ \cJ_1 \geq \bbE[\cJ_1]-2\eta \rho f(\beta) T \big\} \,,
\quad
\cA_2 := \big\{ \cJ_2 \geq  \bbE[\cJ_2]+\eta\rho f(\beta) T \big\} \,,
\]
recalling that \(\bbE[\cJ_1]+\bbE[\cJ_2]=\bbE[\cJ_{(0,T]}^{\gb}] =\rho Tf(\beta)\).
First of all, since \(\cJ_2\) is a non-decreasing function of \(\bar\cU\), the stochastic comparison result \Cref{lem:domination} shows that \(\bbP_{\tau}(\cA_2) \leq \bbP(\cA_2)\).
Therefore using the large deviations~\eqref{LDPoisson} for Poisson variables, since \(\bbE[\cJ_2]\leq \rho f(\gb) T\) we obtain
\[
\bbP_{\tau}(\cA_2) \leq \bbP(\cA_2) \leq \exp\Big( - \frac{\eta^2}{4} \rho  f(\beta) T \Big).
\]
To estimate $\bbP_{\tau}(\cA_1)$, using Chernov's exponential inequality and the product structure of \(\bbP_{\tau}\), see \Cref{lem:sizebiased}, we have 
\begin{equation}
  \label{Chernov1}
\bbP(\cA_1) \leq e^{2\eta^2 \rho f(\beta) T} \prodtwo{1\leq j\leq \cN_{T/2}}{\Delta\tau_j\tf(\beta)\le 1}\bbE_{\tau} \Big[ e^{\eta (\cJ_{(\tau_{j-1},\tau_{j}]}^{\beta}- \bbE[\cJ_{(\tau_{j-1},\tau_{j}]}^{\beta}])} \Big] \,.
\end{equation}
We show below that for any sufficiently small \(\eta>0\), any \(\gb \in (\gb_0,\gb+1)\) and any \(t\leq 1/\tf(\gb)\), we have 
\begin{equation}
  \label{eq:petitLaplace}
\bbE_{t} \Big[ e^{\eta (\cJ_{(0,t]}^{\beta}- \bbE[\cJ_{(0,t]}^{\beta}])} \Big] \leq e^{- \frac14 \eta \rho f(\beta) t} \,,
\end{equation}
recalling that $\dd \bbP_t:= w(0,t,Y) \dd \bbP$, see~\eqref{def:Pt}. 
This implies that 
\[
\bbP(\cA_1) \leq e^{2\eta^2 \rho f(\beta) T} \exp\bigg( - \frac14 \eta \rho f(\beta)  \sum_{j=1}^{\cN_{T/2}}  \Delta\tau_j\ind_{\{\Delta\tau_j  \tf(\beta)\le 1\}} \bigg) \leq e^{- \frac12 \eta^2 \rho f(\beta) T} \,,
\]
where the last inequality holds on the event~\(B\) (say with \(\delta =10\eta\)), concluding the proof of~\eqref{eq:logPtildeA}.
In order to prove~\eqref{eq:petitLaplace}, notice that $\cJ^{\beta}_{(0,t]}$ is an non-decreasing function of $\bar \cU$ and \(x\mapsto e^x -1-x\) is non-decreasing on \(\bbR_+\): we therefore get from \Cref{lem:domination} that
\begin{equation}
  \label{monotoJ}
  \bbE_{t}\big[e^{\eta \cJ_{(0,t]}^{\gb}} -1 - \eta \cJ_{(0,t]}^{\gb}\big] \leq \bbE\big[e^{\eta \cJ_{(0,t]}^{\gb}} -1 - \eta \cJ_{(0,t]}^{\gb}\big] \,,
\end{equation}
and hence 
\begin{equation}
  \label{aftermonotoJ}
  \bbE_{t} \Big[ e^{\eta (\cJ_{(0,t]}^{\beta}- \bbE[\cJ_{(0,t]}^{\beta}])} \Big] \leq 
  \bbE \Big[ e^{\eta (\cJ_{(0,t]}^{\beta}- \bbE[\cJ_{(0,t]}^{\beta}])} \Big] - \eta e^{-\eta \rho f(\beta) t} (\bbE -\bbE_{t})\big[ \cJ_{(0,t]}^{\beta}\big] \,.
\end{equation}
In particular, since \(\rho f(\beta) t \leq c\) by assumption, if \(\eta\) is small enough we get that
\[
\bbE_{t} \Big[ e^{\eta (\cJ_{(0,t]}^{\beta}- \bbE[\cJ_{(0,t]}^{\beta}])} \Big] \leq 1+ \eta^2 \rho f(\beta) t  - \frac23 \eta  (\bbE -\bbE_{t})\big[ \cJ_{(0,t]}^{\beta}\big] \leq e^{\eta^2 \rho f(\beta) t  - \frac23 \eta  (\bbE -\bbE_{t})[ \cJ_{(0,t]}^{\beta}]} \,.
\]
Next we show that for \(t\leq 1/\tf(\beta)\) we have
\begin{equation}
  \label{shiftinE}
   \bbE\big[ \cJ_{(0,t]}^{\gb}\big]-  \bbE_{t}[\cJ^{\beta}_{(0,t]}] \geq \frac12 \rho f(\beta) t \,,
\end{equation}
to conclude the proof of~\eqref{eq:petitLaplace} provided that \(\eta\) is small enough.
Using Mecke's formula \cite[Theorem 4.1]{LP18}, recalling the Poisson construction of Section~\ref{sec:propertiesRW} and using \Cref{lem:sizebiased}, we have
\begin{equation}
  \label{Mecke1}
\bbE_t\big[ \cJ_{(0,t]}^{\gb}\big] = \rho t \sum_{\ell \geq \beta_0 K(1/\tf(\gb))^{-1}} \bar\mu(\ell)  \frac{1}{2\ell+1}\sum_{x=-\ell}^{\ell} \frac{\P(W_t=x)}{\P(W_t=0)} \,.
\end{equation}
Now, using \Cref{monotolemma} we get that \(\P(W_t=0) \geq \P(W_{1/\tf(\gb)}=0)\), so recalling that \(K(s) = \beta_0 \P(W_s=0)\) we get that \( (2\ell+1) \P(W_t=0)\geq 2 \) for \(\ell \geq \beta_0 K(1/\tf(\gb))^{-1}\) and \(t\leq 1/\tf(\gb)\).
We therefore end up with 
\[
\bbE_t\big[ \cJ_{(0,t]}^{\gb}\big] \leq \frac12  \rho t \sum_{\ell \geq \beta_0 K(1/\tf(\gb))^{-1}} \bar\mu(\ell) =\frac12 \rho f(\beta) t \,,
\]
which proves~\eqref{shiftinE} and concludes the proof.
\end{proof}

\begin{proof}[Proof of \eqref{eq:logPB}]
First of all, since $B$ is measurable w.r.t.\ $\sigma(\tau\cap[0,\frac12T])$, we can remove the conditioning at the expense of an harmless multiplicative constant \(C\), see \Cref{lem:removecond}.
We therefore only need to show that \(\bQ_{\beta}( B^{\cc}) \leq \exp(- c_0 \delta \tf(\gb) T)\) for all \(T\) large.
Let us set \(\hat\Delta_j := \Delta\tau_j \ind_{\{\Delta\tau_j \leq 1/\tf(\gb)\}}\), so that we can write
\begin{equation*}
  \bQ_{\beta} \left( B^{\cc}\right)  \leq \bQ_{\gb} \Big( \cN_{T/2} \leq  S  \Big) + \bQ_{\gb} \Big( \sum_{j=1}^{S } \hat\Delta_j <  \delta T \Big) \,.
\end{equation*}
Therefore setting $m_{\gb} = \bQ_{\gb}[\tau_1]$ and $S:=  T/(4 m_{\beta})$, we show the following: There is a constant \(c_0\) such that, if \(\delta\) is small enough
\begin{equation}
  \label{reduc-badcaz}
   \bQ_{\gb} \big(  \tau_{S} \geq   2 m_{\beta} S \big) \leq e^{-  c_0  \tf(\gb) m_{\gb} S}\,, \quad
    \bQ_{\gb} \Big( \sum_{j=1}^{S} \hat\Delta_j < 4\delta m_{\gb} S \Big) \leq e^{-c_0 \tf(\gb) m_{\gb} S} \,,
\end{equation}
for all $T$ sufficiently large.
For the first inequality, using Chernov's bound, we get that for $u>0$,
\[
   \bQ_{\gb} \big(  \tau_{S} \geq  2 m_{\gb} S \big) \leq e^{- u \tf(\beta) \, 2 m_{\gb} S} \bQ_{\gb}\big[ e^{ u \tf(\beta) \tau_1} \big]^{S}  \leq e^{- \frac12 u \tf(\beta) \, 2 m_{\gb} S}
\]
where for the last inequality we have used \Cref{lem:Laplace} to get that $\bQ_{\gb}[ e^{ u \tf(\beta) \tau_1}] \le e^{\frac{3}{2} u \tf(\beta) m_{\beta}}$ for \(u\) small enough.
For the second inequality in~\eqref{reduc-badcaz}, using again Chernov's bound, we have 
\[
  \bQ_{\gb} \Big( \sum_{j=1}^{S} \hat\Delta_j < 4 \delta m_{\gb} S \Big) \leq e^{\tf(\gb) 4 \delta m_{\gb} S} \bQ_{\gb}\big[ e^{-\tf(\beta) \hat\Delta_1} \big]^{S}.
\]
Since $\tf(\gb)\hat\Delta_1 \leq 1$ we have 
\(
\bQ_{\gb}\big[ e^{-\tf(\gb) \hat\Delta_1} \big] \leq 1 - \frac{1}{2}\tf(\gb) \bQ_{\gb}[\hat\Delta_1] \leq e^{- \frac12 c m_{\gb} \tf(\gb)} \,,
\)
where for the last inequality we have used \Cref{lem:mbeta} to get that $\bQ_{\gb}[\hat\Delta_1] \geq c m_{\gb}$. 
Altogether, provided that $4\delta \leq \frac14 c$, we second inequality in~\eqref{reduc-badcaz}.
This concludes the proof of~\eqref{eq:logPB}.
\end{proof}

\subsection{Proof of \texorpdfstring{\Cref{prop:martingale}}{the proposition}}

 Recalling \Cref{lem:StefHub}, our first step is to  introduce the events \(\cA\) and \(B\) that we use in \eqref{eq:boundtildePAT} to prove our result.
We recall the Poisson construction of Section~\ref{sec:propertiesRW}, and consider the following random variable defined for \(0\leq a <b \leq T\)
\begin{equation}
  \label{def:Fab}
  F_{(a,b]} := \sum_{i: \vartheta_i \in (a,b]} \ind_{\{U_i K(\vartheta_i)\ge \beta_0\}}  \,.
\end{equation}
We then introduce the following associated event, for some fixed small \(\eta>0\) 
\[
  \cA = \cA_T := \big\{ F_{(0,T]} - \bbE[F_{(0,T]}] \leq -   \eta \rho \log T\big\} \,.
\]
Let us also introduce, for some (small) parameter \(\delta\) the event \(B\) as 
\[
B = \bigg\{ \sum_{j=1}^{\cN_{T/2}}  \frac{\Delta \tau_j}{1\vee \tau_j }  \geq \delta \log T \bigg\} \,.
\]
Then, in view of \Cref{lem:StefHub} and \eqref{eq:boundtildePAT} we only need to show that there is a constant \(c_0>0\) such that, if \(\delta\) is fixed small enough, for any \(\eta\) small enough and \(\delta= \eta^{1/2}\), for all large \(T\) we have
\begin{equation}
  \label{eq:polPA}
\bbP(\cA) \leq T^{- c_0 \eta^2 \rho} \,, 
\quad
\bbP_{\tau}(\cA^{\cc}) \ind_{B} \leq T^{- c_0 \eta^2 \rho} 
\quad \text{ and } \quad
\bQ_{\beta_0,T}(B^{\cc}) \leq  T^{- c_0 \delta} \,.
\end{equation}

\begin{proof}[Proof of~\eqref{eq:polPA} for \(\bbP(\cA)\)]
Let us notice that \(F_{(0,T]}\) is a Poisson random variable with mean given by (applying Mecke's formula):
\begin{equation}\label{eft}
 \bbE[F_{(0,T]}]=\rho \int^T_0 \sum_{\ell \ge \beta_0 K(t)^{-1}} \bar \mu(\ell) \dd t \,.
\end{equation}
Recalling~\eqref{lassum} and~\eqref{implicit}, we have
\(
\sum_{\ell \ge \beta_0 K(t)^{-1}} \bar \mu(\ell) \stackrel{t\to \infty}{\sim} c_J \, t^{-1}, 
\) so that combined with \eqref{eft} we get  
\begin{equation*}
 \bbE[F_{(0,T]}] \stackrel{T\to \infty}{\sim} c_{J} \rho \log T \,.
\end{equation*}
Using large deviations for Poisson random variables, see~\eqref{LDPoisson}, we obtain that \(\bbP(\cA) \leq e^{- c_0 \eta^2 \rho \log T}\), which gives the desired bound.
\end{proof}

\begin{proof}[Proof of~\eqref{eq:polPA} for \(\bbP_{\tau}(\cA^{\cc}) \ind_B\)]
Let us decompose \(F_{(0,T]} = F_1+F_2\), with 
\[
 F_1 := \sum_{j=1}^{\cN_{T/2}} F^{(j)} \qquad \text{ with } F^{(j)} = F_{( \frac{\tau_{j-1}+\tau_j}{2},\tau_j]} \,,
\]
and we let \(F_2= F_{(0,T]} -F_1\).
We then have \(\bbP_{\tau}(\cA^{\cc}) \leq \bbP_{\tau}(\cA_1)+ \bbP_{\tau}(\cA_2)\), with 
\[
\cA_1 = \big\{ F_1 - \bbE[F_1] \ge - 2\eta \rho \log T\big\}\,,
\qquad
\cA_2 = \big\{ F_2 - \bbE[F_2] \geq \eta  \rho \log T\big\}\,.
\] 
Since \(F_2\) is a non-decreasing function of \(\bar \cU\), we can use the stochastic domination of \Cref{lem:domination} to get that 
$\bbP_{\tau}(\cA_2) \leq \bbP(\cA_2)$.
Then, since $F_2$ is a Poisson variable under \(\bbP\), whose mean is smaller than $\bbE[F_{(0,T]}]\stackrel{T\to\infty}{\sim} c_{J} \rho \log T$, the large deviation~\eqref{LDPoisson} gives that \(\bbP(\cA_2) \leq e^{- c \eta^2 \rho \log T}\), as desired.
For \(\bbP_{\tau}(\cA_1)\), using Chernov's bound and the product structure of \(\bbP_{\tau}\) (see \Cref{lem:sizebiased}), we get similarly to~\eqref{Chernov1} that
\begin{equation}
  \label{Chernov2}
  \bbP_{\tau}(\cA_1) \leq e^{2 \eta^2 \rho \log T}\prodtwo{1\leq j\leq \cN_{T/2}}{\Delta\tau_j\tf(\beta)\le 1} \bbE_{\tau} \Big[ e^{\eta (F^{(j)}- \bbE[F^{(j)}])} \Big] \,.
\end{equation}
Now, as in \eqref{monotoJ}-\eqref{aftermonotoJ}, using the stochastic comparison  of \Cref{lem:domination} we have that
\begin{equation}
  \label{monotoF}
 \bbE_{\tau}\Big[ e^{\eta (F^{(j)}-\bbE[F^{(j)}]) } \Big]\le\bbE_{\tau}\Big[ e^{\eta (F^{(j)}-\bbE[F^{(j)}]) } \Big] - \eta e^{- \eta \bbE[F^{(j)}]}\, (\bbE-\bbE_{\tau})[F^{(j)}].
\end{equation}
Note that using Mecke's formula as in~\eqref{eft}, we have that 
\[
\bbE[F^{(j)}] = \rho \int_{\frac{\tau_{j-1}+\tau_j}{2}}^{\tau_{j}} \sum_{\ell \geq \beta_0 K(t)^{-1}} \bar \mu(\ell) \dd t  \,.
\]
Using \Cref{monotolemma}, we get that \(K(\tau_j) \leq K(t) \leq K(\tau_j/2)\) in the integral above. 
Hence, recalling that \(\sum_{\ell\geq \beta_0 K(s)^{-1}} \bar \mu(\ell) \stackrel{s\to \infty}\sim c s^{-1}\), we thus have that \(F^{(j)}\) is a Poisson random variable with mean
\begin{equation}
  \label{meanFj}
c \,\rho \frac{\Delta \tau_j}{1\vee \tau_j }   \leq \bbE[F^{(j)}] \leq   c' \, \rho \frac{\Delta \tau_j}{1\vee \tau_j} \,,
\end{equation}
where \(c,c'\) are universal constants.
In particular, \(\bbE[F^{(j)}] \leq c'\), so from~\eqref{monotoF} we get that for \(\eta\) small enough
\begin{equation}
  \label{eq:shiftF}
  \bbE_{\tau}\Big[ e^{\eta (F^{(j)}-\bbE[F^{(j)}]) } \Big] \leq 1+ \eta^2 \bbE[F^{(j)}] - \frac23 \eta (\bbE-\bbE_{\tau})[F^{(j)}] \,.
\end{equation}
Using Mecke's formula as in \eqref{Mecke1}, we obtain that 
\begin{equation*}
 \bbE_{\tau}[F^{(j)}] =  \rho \int_{\frac{\tau_{j-1}+\tau_j}{2}}^{\tau_{j}} \sum_{\ell \geq \beta_0 K(t)^{-1}} \bar\mu(\ell)  \frac{1}{2\ell+1}\sum_{x=-\ell}^{\ell} \frac{\P(W_t=x)}{\P(W_t=0)} \dd t  \leq \frac12 \rho \int_{\frac{\tau_{j-1}+\tau_j}{2}}^{\tau_{j}} \sum_{\ell \geq \beta_0 K(t)^{-1}} \bar\mu(\ell)  \dd t \,,
\end{equation*}
where the inequality holds because \((2\ell+1)\P(W_t=0) = (2\ell+1) \beta_0^{-1} K(t) \geq 2\).
We therefore get that \(\bbE_{\tau}[F^{(j)}]\leq \frac12 \bbE[F^{(j)}]\), which plugged in~\eqref{eq:shiftF} gives that for \(\eta\) sufficiently small
\[
\bbE_{\tau}\Big[ e^{\eta (F^{(j)}-\bbE[F^{(j)}]) } \Big] \leq e^{ -\frac14 \eta \bbE[F^{(j)}]}  \,.
\]
Going back to~\eqref{Chernov2}, we therefore get
\[
\bbP_{\tau}(\cA_1) \leq e^{2 \eta^2 \rho \log T} \exp\Big(-\frac14 \sum_{j=1}^{\cN_{T/2}} \bbE[F^{(j)}]\Big) \leq e^{2 \eta^2 \rho \log T - `\frac{c}{4} \eta \delta \rho \log T  }
\]
where the last inequality holds on the event \(B\), recalling that \(\bbE[F^{(j)}] \geq c \rho \frac{\Delta\tau_j}{1\vee\tau_j}\), see~\eqref{meanFj}.
Taking \(\delta = \eta^{1/2}\) with \(\eta\) small enough shows that \(\bbP_{\tau}(\cA_1) \ind_B \leq e^{- c' \eta^{3/2} \rho \log T }\), giving the desired bound.
\end{proof}

\begin{proof}[Proof of~\eqref{eq:polPA} for \(\bQ_{\beta_0,T}(B^\cc)\)]
Since~\(B\) depend only on \(\tau\cap [0,T/2]\), we can again use \Cref{lem:removecond} to remove the conditioning, at the expense of a harmless multiplicative constant \(C\).
Recall that \(\bQ_{\beta_0}=\bQ\): we need to show that \(\bQ(B^{\cc}) \leq T^{-c_0 \delta}\).
Letting $R_t:=\min (\tau \cap [t,\infty))$ denote the next renewal point after $t$, notice that
\begin{equation*}
 \sum_{j=1}^{\cN_{T/2}} \frac{\Delta \tau_j}{1\vee \tau_j} = \int^{\infty}_0 \frac{\ind_{\{R_t\le T/2\}}}{1\vee R_t} \dd t \,.
\end{equation*}
Then, for \(k\geq 1\), define the stopping time $S_k$ and the event $D_k$ as follows
\[
S_k:= \inf \{ j : \Delta \tau_j\ge 2^{k}\}  \text{ and } D_k:= \{ \Delta \tau_{S_k} \le   2^{k+1} \} \,.
\]
Note that the events $D_k$ are independent under $\bQ$, and that we have
\(D_k\subset \{ \forall t\in [2^{k-1},2^k],\,  R_t\le 4 t\}\).
As a result we have 
\[
\int^{\infty}_0 \frac{\ind_{\{R_t\le T/2\}}}{1\vee R_t} \dd t \ge   \sum_{k=1}^{\log_2 (\frac{T}{8})}\left(\int^{2k}_{2^{k-1}}\frac{\dd t}{4t} \ind_{D_k}\right) \ge \frac{\log 2}{4} \sum_{k=1}^{\log_2 (\frac{T}{8})} \ind_{D_k} \,.
\]
Now, since \(\bQ( D_k) = \bQ(\tau_1\leq 2^{k+1} \mid \tau_1 \geq 2^{k})\) we get that \(\lim_{k\to\infty} \bQ(D_k) = 1-2^{-\alpha}\). In particular 
\[
\sum_{k=1}^{\log_2 (\frac{T}{10})} \bQ(D_k) \stackrel{T\to\infty}{\sim} c_{\alpha} \log T \,.
\]
Therefore, provided that \(\delta\) has been fixed small enough and \(T\) is sufficiently large, we get that 
\[
\bQ(B^{\cc}) \leq \bQ\bigg(\sum_{k=1}^{\log_2 (\frac{T}{10})} \ind_{D_k} \leq \frac{3\delta}{\log 2} \log T\bigg) \leq \bQ\bigg(\sum_{k=1}^{\log_2 (\frac{T}{10})} (\ind_{D_k} - \bQ(D_k)) \leq -  \frac12 c_{\alpha} \log T\bigg) \,.
\]
Applying Hoeffding's inequality, one concludes that \(\bQ(B^{\cc}) \leq e^{- c \log T}\), as desired.
\end{proof}

\section{Fractional moment, coarse-graining and change of measure}
\label{sec:method}

We explain in this section the method that we use to prove  that $\beta_c(\rho)>\beta_0$ and derive lower bounds on $\beta_c(\rho) - \beta_0$ (or \(\gb_c(\rho)/\gb_0\)).
The idea introduced in \cite{DGLT09} is by now classical and has been first implemented for the RWPM in~\cite{BS10}.
Our approach is similar to that of~\cite{BS10}, but we provide the details for completeness.

\subsection{The fractional moment and coarse-graining method}

We let $T>0$ be a fixed real number and consider the (free)partition function of a system who whose length is an integer multiple of~$T$.
Using Jensen's inequality, we obtain that for any $\theta\in (0,1)$
\begin{equation}
  \label{fracmomo}
 \tf(\beta,\rho)
  % = \lim_{n\to \infty} \frac{1}{ nT} \bbE \big[\log Z^{Y,\mathrm{c}}_{\beta,nT}\big]
   = \lim_{n\to \infty} \frac{1}{\theta nT}\bbE \left[\log (Z^{Y}_{\beta,nT})^{\theta}\right] \le \liminf_{n\to \infty} \frac{1}{\theta n T} \log \bbE\left[(Z^{Y}_{\beta,nT})^{\theta} \right] \,.
\end{equation}
The value of $\theta$ is mostly irrelevant for our proof, but must satisfy $(1+\alpha)\theta>1$ with \(\alpha = \frac{\gamma}{1-\gamma}\) from~\eqref{formkt} (for instance one may take $\theta=(1+\alpha)^{-1/2}$).
Note that we need here to take the fractional moment \(\bbE[Z^{\theta}]\) instead of the truncated moment \(\bbE[Z\wedge 1]\) as in Section~\ref{sec:upper}, because we want to exploit a quasi-multiplicative structure of the model, which does not behave well with truncations.
Concerning the value of $T$, we consider it to be equal to $1/\tf(\beta)$ ,which corresponds to the \textit{correlation length} of the annealed system.
We want to prove that $\tf(\rho,\beta)=0$, for some values of $\beta$ and $\rho$.

Hence in view of \eqref{fracmomo} it is sufficient to show that for these values of $\rho$,  $\bbE[(Z^{Y}_{\beta,nT})^{\theta}]$ is bounded uniformly in $n$.
For this, we perform a coarse-graining procedure.
We divide the system into segments of length \(T\) of the form $[(i-1)T,iT]$, which we refer to as \textit{blocks}, and we decompose the partition function according to the contribution of each block.
More precisely, we split the integral~\eqref{goodexpress} according to the set of blocks visited by $\{t_1,\dots,t_k\}$.
For an arbitrary $k\geq 0$ and $\bt \in \cX_k(nT)$, we define $I(\bt)$ the set of blocks visited by $\bt$, that is
\begin{equation*}
 I(\bt) = \Big\{ i \, : \,  \{t_1,\dots, t_k\} \cap ((i-1)T,iT] \ne \emptyset\Big\}
\end{equation*}
Then, letting $I$ encode the set of visited blocks, we can write
\begin{equation}\label{decompz}
  Z^{Y}_{\beta,nT} =: \sum_{I\subset \lint n \rint } Z^{Y}_{\beta,T,I},
\end{equation}
where $Z^{Y}_{\beta,T,\emptyset}:= (\beta/\beta_0) K_w(0,nT,Y)$ and for $|I|\ge 1$, $Z^{Y}_{\beta,T,I}$ is obtained by restricting the integrals \eqref{goodexpress} to the sets
\[
\cX_k(T,I):= \Big\{ {\bf t} \in\bigcup_{k\ge 0} \cX_k(nT) \ : \ I({\bf t})=I \Big\} \,,
\]
that is
\begin{equation*}
Z^{Y}_{\beta,T,I}:=\sum^{\infty}_{k=0} \left(\beta/\beta_0 \right)^{k} \int_{\cX_k(T,I)}\prod_{i=1}^k  K_w(t_{i-1},t_i,Y) \dd t_i
\end{equation*}
Let us now rewrite the above expression in a more explicit way. 
Integrating over all $t_i$ within a block except for the first one, we obtain that for $I=\{i_1,\dots, i_{\ell}\}$ with $1\leq i_1 < \cdots <i_{\ell} $, setting $s_0=0$ by convention we have
\begin{equation}
\label{blockdecompo}
Z^{Y}_{\beta,T,I}=   \!\!\!\!\! \int\limits_{ \substack{ (r_{j},s_{j})_{j=1}^{\ell} \\ (i_j-1)T < r_j\le s_j\le i_{j}T } } \!\!\!\!\!   \Big(\frac{\beta}{\beta_0}\Big)^{\ell} \prod_{j=1}^{\ell} K_w(s_{{j-1}},r_{j},Y) Z^{Y}_{\beta,[r_{j},s_{j}]}\dd r_j(\delta_{r_j}(\dd s_j)+ \dd s_j) \,,
\end{equation}
where for $r< s$, we have defined the constrained partition function on the segment \([r,s]\) by setting 
\begin{equation}\label{partisegment}
 Z^{Y}_{\beta,[r,s]}:= \beta \bE\left[ e^{\beta \int^s_r  \ind_{\{X_{t}=Y_t\}} \dd t}  \ind_{\{X_{s}=Y_s\}} \ | \ X_r=Y_r\right] 
\end{equation}
and set  $Z^{Y}_{\beta,[s,s]}=1$.
Note that in~\eqref{blockdecompo}, the Dirac mass terms $\delta_{r_i}(\dd s_i)$ is present to take into account the possibility that a given block is visited only once. 
To estimate $\bbE[(Z^{Y}_{\beta,nT})^{\theta}]$,
we combine \eqref{decompz}, together with the inequality $\left(\sum a_i\right)^{\theta}\le \sum a^{\theta}_i$ (valid for any collection of positive numbers) and obtain the following upper bound
\begin{equation}
  \label{coarsegraining}
\bbE\left[ (Z^{Y}_{\beta,nT})^{\theta}\right]\le \sum_{I\subset \lint n \rint } \bbE\left[\left(Z^{Y}_{\beta,T,I}\right)^{\theta}\right] \,.
\end{equation}
% where we have defined
% $\nu_\rho:= \max_{t>0} \frac{K((1-\rho)t)}{K(t)}$ (note that $\nu$ is finite and increasing in $\rho$) and
% \begin{equation}
%   \bar Z^{Y}_{\beta,T,I}:= \nu_{\rho}^{\ell +1}\!\!\!\!\!\!\!\!\!\!\!\int\limits_{ \substack{ (r_{j},s_{j})_{j=1}^{\ell} \\ (i_j-1)T < r_j\le s_j\le i_{j}T } } \!\!\!\!\!\!\!\!\!\!\!\!\!\!\!\!\!\! \Big(\frac{\beta}{\beta_0}\Big)^{\ell}  K(n-s_\ell)\prod_{j=1}^{\ell} K(r_j,s_{{j-1}}) Z^{Y}_{\beta,[r_{j},s_{j}]}\dd r_j(\delta_{r_j}(\dd s_j)+ \dd s_j) .
% \end{equation}
% The purpose of working with   $\bar Z^{Y}_{\beta,T,I}$ is to cut dependence between blocks.

\subsection{Change of measure argument and further reduction}

The idea behind   \eqref{coarsegraining} is to reduce our proof to an estimate for each visited block in $I$.
For this, we fix a function~$g_I$ of the enriched random environment $\cU$ and we use H\"older's inequality to obtain
\begin{equation}
  \label{eq:Holder}
 \bbE\left[\left( Z^{Y}_{\beta,T,I}\right)^{\theta}\right]= \bbE\left[\left(g_I(\cU) Z^{Y}_{\beta,T,I}\right)^{\theta}g_I(\cU)^{-\theta}\right] \le \bbE\left[ g_I(\cU)  Z^{Y}_{\beta,T,I}\right]^{\theta} \bbE\Big[ g_I(\cU)^{-\frac{\theta}{1-\theta}}\Big]^{1-\theta}.
\end{equation}
We want $g_I$ to penalize the trajectories $Y$ that contribute most to the expectation.
The penalization we introduce is only based the process $\cU$ restricted to the visited blocks.
For this we introduce an event $\cA\in \cF_{[0,T]}$ meant to be a rare set of favorable environment within the first block (the precise requirement will be~\eqref{singlecell}-\eqref{singlecellis} below). 
We then consider a function $g_I$ which penalizes blocks whose environment is favorable, that is
$g_I(\cU)=\prod_{i\in I} g_i(\cU)$
% where the shift operator  $\theta_s$ is defined by $\theta_s(\cU):=\cU-(0,0,s)$,
% where we have defined $Y_{[a,b]} := (Y_s-Y_a)_{s\in [a,b]}$.
% In other words, $g_1=g$ is a function of the trajectory in the first block $(Y_s)_{s\in[0,T]}$, and $g_i$ is the function of the $Y$ in the $i$-th block obtained by shifting the increments of the random walk.
with
\begin{equation}
  \label{def:g}
 g_i(\cU)= g(\theta_{iT}\cU)  \quad \text{ and } \quad  g:= \ind_{\cA^{\cc}} + \eta \ind_{\cA} \quad
\end{equation}
where $\theta_{iT} \cU:= \cU-(0,0,iT)$ is the shifted point process and $\eta:= \bbP(\cU\in\cA)^{\theta/(1-\theta)}$ (the value of \(\eta\) is chosen for convenience, see \eqref{eq:convenience} just below).
Note that because $\cA\in \cF_{[0,T]}$, the variables $g_{i}(\cU)$ are i.i.d.\
In particular, thanks to the definition of \(\eta\),  we directly have that, for any $I$,
\begin{equation}
  \label{eq:convenience}
  \bbE\left[ g_I(\cU)^{-\frac{\theta}{1-\theta}}\right]^{1-\theta} = \bbE\left[g(\cU)^{-(1-\theta)/\theta}\right]^{(1-\theta)|I|} =  \big( \bbP(\cU\in\cA^{\cc}) +1 \big)^{ (1-\theta)|I|} \leq 2^{|I|}\,.
\end{equation}
Hence, thanks to~\eqref{eq:Holder}, the inequality~\eqref{coarsegraining} becomes
\begin{equation}
  \label{coarsegraining2}
  \bbE\left[ (Z^{Y}_{\beta,nT})^{\theta}\right]\le \sum_{I\subset \llbracket n \rrbracket} 2^{|I|} \bbE\left[ g_I(\cU) Z^{Y}_{\beta,T,I}\right]^{\theta} \,.
\end{equation}
From now on, for simplicity, let us write $G_I:=g_{I}(\cU)$, $G_i:=g_i(\cU)$ and $G:=g(\cU)$.
Using the block decomposition~\eqref{blockdecompo} and Fubini's theorem, we have
\begin{equation}\label{struck2}
\bbE\left[ G_I Z^{Y}_{\beta,T,I}\right]
 \le  \!\!\!\!\!
\int\limits_{ \substack{ (r_{j},s_{j})_{j=1}^{\ell} \\ (i_j-1)T < r_j\le s_j\le i_{j}T } } \!\!\!\!\! \Big(\frac{\beta}{\beta_0}\Big)^{\ell} \bbE\left[ \prod_{j=1}^{\ell} K_w(s_{{j-1}},r_j,Y) G_{i_j}  Z^{Y}_{\beta,[r_{j},s_{j}]}\dd r_j(\delta_{r_j}(\dd s_j)+ \dd s_j) \right].
\end{equation}
Since the $(G_{i_j}  Z^{Y}_{\beta,[r_{j},s_{j}]})^{\ell}_{j=1}$ are independent, it may be convenient to replace  $\prod_{j=1}^{\ell+1} K_w(s_{{j-1}},r_j,Y)$ in the expectation above by a deterministic upper bound in order to factorize of the expectation.
Using \Cref{lem:unimod}, we have for any $\rho\in (0,1/2)$
\[
\frac{K_w(s,t,Y)}{K(t-s)} = \frac{ \bP(X_{t-s} = Y_t-Y_s)}{\P(W_{t-s}=0)}\leq  \frac{\P(W_{(1-\rho)(t-s)})}{\P(W_{t-s}=0)
} \leq \sup_{r\ge 0} \frac{K(r/2)}{K(r)}  \,.
\]
Since $K(r/2)/K(r)$ is continuous and converges to $1$ and $2^{1+\alpha}$ at $0$ and $\infty$ the r.h.s.\ is finite. 
Hence there exists some constant $C$ such that for all  $\rho\in (0,1/2)$ and $\beta\in (\beta_0,2\beta_0)$ we have
\begin{equation}
  \label{struck3}
\bbE\left[ g_I(Y) Z^{Y}_{\beta,T,I}\right] \leq C^{\ell}\!\!\!\!\!
\int\limits_{ \substack{ (r_{j},s_{j})_{j=1}^{\ell} \\ (i_j-1)T < r_j\le s_j\le i_{j}T } }\!\!\!\!\!  \prod_{j=1}^{\ell} K(r_j-s_{{j-1}}) \bbE\left[G_{i_j}  Z^{Y}_{\beta,[r_{j},s_{j}]}\right]\dd r_j(\delta_{r_j}(\dd s_j)+ \dd s_j) \,.
\end{equation}
Let us stress that while the a variant of \eqref{struck3} may be valid for $\rho$ close to $1$, it would involve a constant $C$ that depends on $\rho$. 
For this reason, to prove \Cref{prop:largerho} we rely on~\eqref{struck2} and use another trick to perform factorization. For all other results we use \eqref{struck3}.
In all cases, the main task is to chose an event $\cA$ (recall the definition \eqref{def:g} of \(g\)) which has small probability but makes  but that carries most of the expectation of $Z^{Y}_{\beta,[r,s]}$ for most choices of $r$ and $s$ in the intervals considered.

Let us now explain how one can evaluate $\bbE[G_{i_j}  Z^{Y}_{\beta,[r_{j},s_{j}]}]$, we also apply the same idea for the expectation present in \eqref{struck2}.
By translation invariance it is sufficient to consider the case of \(\bbE[G Z^{Y}_{\beta,[r,s]}]\).
Taking the convention $t_0=r$, $t_{k+1}=s$ and recalling~\eqref{goodexpress} and the definition~\eqref{def:w} of \(w(s,t,Y)\), we have
\begin{equation*}
 \bbE\big[G Z^{Y}_{\beta,[r,s]}\big]=  \sum^{\infty}_{k=0} \left(\beta/\beta_0 \right)^{k+1}  \int_{\cX_k([r,s])}\bbE\Big[ G \prod_{i=1}^{k+1} w (t_{i-1},t_i,Y)  \Big]   \prod_{i=1}^{k+1} K(t_i-t_{i-1})\prod_{i=1}^k\dd t_i.
\end{equation*}
Recalling also the definition~\eqref{def:Pt} of the weighted measure \(\bbP_{\bt}\), we can simply rewrite 
\[
\bbE\Big[ G \prod_{i=1}^{k+1} w (t_{i-1},t_i,Y) \Big] = \bbE_{\bt}[G] \,,
\] 
where $\bt = (t_i)_{0\leq i\leq k+1}$ with $t_0=r$ and $t_{k+1} =s$.
We can now interpret the above expression as the partition function of a pinning model based on the renewal process~\(\tau\) introduced in Section~\ref{sec:annealed}.
Let $\bQ_{[r,s]}$ be the law of the renewal process \(\tau\) with pinned boundary condition \(r,s \in \tau\). 
More precisely, \(\bQ_{[r,s]}\) is the probability on $\bigsqcup_{k=0}^{\infty} \{r\}\times\cX_k([r,s])\times\{s\}$, whose density on $\{r\}\times\cX_k([r,s])\times\{s\}$ w.r.t.\ the Lebesgue measure is given by
$u(s-r)^{-1} \prod_{i=1}^{k+1} K(t_i-t_{i-1})$, which corresponds to the law of \(\tau\cap [r,s]\) under \(\bQ(\cdot \mid r,s\in \tau)\).
We then have that 
\begin{equation}
  \label{replaceterms}
  \begin{split}
  \bbE[G Z_{\gb,[r,s]}^{Y}]  &= u(s-r) \bQ_{[r,s]}\Big[ (\gb/\gb_0)^{|\tau|}\bbE_{\tau}[G]  \Big] \\
  & \le  u(s-r) \bQ_{[r,s]}\Big[ (\gb/\gb_0)^{2|\tau|}\Big]^{1/2} \bQ_{[r,s]} \big[ \bbE_{\tau}[G]^2 \big]^{1/2} \,.
  \end{split}
\end{equation}
The second line is obtained using Cauchy-Schwarz and its objective is to decouple the effect of $G$ and that of the pinning reward. 
Now, simply writing \(\beta' =\beta^2/\beta_0\) and recalling~\eqref{eq:zpinning}, we have that 
\(
\bQ_{[r,s]}[ (\gb/\gb_0)^{2|\tau|}] = z^{\cons}_{\beta',r-s} / u(r-s)\,.
\)
Since by assumption $s-r\le T =\tf(\beta)^{-1} \leq C \tf(\gb')^{-1}$, we get from~\eqref{eq:homogeneous} (or \cite[Lemma~3.1]{BLirrel}) that this is bounded by a constant.
All together, we deduce from~\eqref{replaceterms} that
\begin{equation}\label{cku}
   \bbE\big[G Z_{\gb,[r,s]}^{Y}\big] \le C   u(s-r) \bQ_{[r,s]}\big[ \bbE_{\tau}[G]^2 \big]^{1/2}.
\end{equation}

\subsection{Finite-volume criterion and good choice of event \texorpdfstring{$\cA$}{A}}

Let us now provide a finite-volume criterion that ensures that \(\tf(\beta,\rho)=0\) in terms of the existence of an event \(\cA\) with specific properties.
Recall that we have fixed $T:=\tf(\beta)^{-1}$.
We say that an event $\cA\in \cF_{[0,T]}$ is $\gep$-good
if it satisfies the following:
\begin{equation}
  \label{singlecell}
  \begin{split}
   &\bbP(\cA)\leq \gep \,, \\
\forall (r,s) \subset [0,T]^2,  \quad &  \left( s-r\ge \gep T\right) \  \Rightarrow  \  \bQ_{[r,s]}\big(  \bbP_{\tau}(\cA^{\cc})\big)\le \gep\,.
  \end{split}
\end{equation}

\begin{proposition}
  \label{prop:coarse-graining}
There exists $\gep>0$ such that for any \(\rho \in (0,\frac12]\) and $\beta\in [\beta_0,2\beta_0]$, the existence of some $\gep$-good event implies that $\tf(\beta,\rho)=0$.
\end{proposition}

For the case \(\rho \in (\frac12,1)\), we need to include in the definition of \(\gep\)-goodness an additional requirement that will allow for  factorization. 
We say that an event $\cA\in \cF_{[0,T]}$ is $\gep$-better
if it satisfies the following:
\begin{equation}\label{singlecellis}
  \begin{split}
   &\bbP(\cA)\leq \gep \,, \\
\forall (r,s) \subset [0,T]^2,  \quad  &\left( s-r\ge \gep T\right) \  \Rightarrow  \  \bQ_{[r,s]}\big(  \bbP_{\tau}(\cA^{\cc}  \ | \ \cF_{\bbR\setminus[r,s]}) \big)\le \gep.
  \end{split}
\end{equation}

\begin{proposition}
\label{prop:coarse-graining2}
There exists $\gep>0$ such that for any \(\rho \in (0,1)\) and $\beta\in [\beta_0,2\beta_0]$, the existence of some $\gep$-better event implies that $\tf(\beta,\rho)=0$.
\end{proposition}

\begin{proof}[Proof of Proposition~\ref{prop:coarse-graining}]
Let us assume that $\cA$ in the construction~\eqref{def:g} of $g$ is $\gep$-good.
If we combine~\eqref{struck3} and \eqref{cku}, we have for some $C>0$ that \(\bbE[ G_I Z^{Y}_{\beta,T,I}]\) is bounded by
\begin{equation}\label{srhg}
(C)^{\ell} 
\!\!\!\!\! \int\limits_{ \substack{ (r_{j},s_{j})_{j=1}^{\ell} \\ (i_j-1)T < r_j\le s_j\le  i_{j}T } } \!\!\!\!\!  \prod_{j=1}^{\ell} K(r_j-s_{{j-1}})u(s_j-r_j) \bQ_{[r_{j},s_j]} \Big[\bbE_{\tau}[G_{i_j}]^2 \Big]^{1/2}\dd r_j(\delta_{r_j}(\dd s_j)+ \dd s_j) .
\end{equation}
Now, recalling the definition~\eqref{def:g} of \(g\), the \(\gep\)-good assumption~\eqref{singlecell} implies that 
\begin{equation*}
  \bQ_{[r_j,s_j]} \Big[ \bbE_{\tau}[G_j]^2 \Big]^{1/2}  \leq  \bQ_{[r_j,s_j]} \Big[ \bbE_{\tau}\big[ G_j^2\big] \Big]^{1/2} \le \ind_{\{s_j-r_j\le \gep T\}}+ \left(\gep+\eta^2\right)^{1/2}\,,
\end{equation*}
with \(\eta \leq \gep^{\theta/(1-\theta)}\).
Now, using the regular variation of \(K(\cdot)\) and \(u(\cdot)\), see~\eqref{formkt} and \eqref{renewalr}-\eqref{renewal}, we see that there is a constant \(C>0\) such that, for any \(a < 0\) and \(b\geq T\) and any \( \gep\in (0,1)\)
\begin{equation}
  \label{eq:iterationstep}
  \begin{split}
  \int\limits_{ \substack{ 0<r<s<T \\ s-r \leq \gep T}} K(r-a)u(s-r) K(b-s) \dd r \dd s 
  & \leq C \gep^{\alpha\wedge 1} \int\limits_{0<r<s<T } K(r-a)u(s-r) K(b-s) \dd r \dd s \\
  \int\limits_{ \substack{ 0<r<s<T \\ s-r \leq \gep T}} K(r-a)u(s-r)\dd r \dd s 
  & \leq C \gep^{\alpha\wedge 1} \int\limits_{0<r<s<T} K(r-a)u(s-r) \dd r \dd s \,.
  \end{split}
\end{equation}
The proof is left to the reader (it follows that of \cite[Equation~(6.7)]{BS11}).
Hence, going back to~\eqref{srhg} and applying~\eqref{eq:iterationstep}, we get that 
\begin{equation}
  \label{eq:boundPI}
\bbE[ G_I Z^{Y}_{\beta,T,I}] \leq (\delta)^{|I|}  \!\!\!\!\!\int\limits_{ \substack{ (r_{j},s_{j})_{j=1}^{\ell} \\ (i_j-1)T < r_j\le s_j\le  i_{j}T } } \!\!\!\!\!  \prod_{j=1}^{\ell} K(r_j-s_{{j-1}})u(s_j-r_j)\dd r_j(\delta_{r_j}(\dd s_j)+ \dd s_j) \,,
\end{equation}
with \(\delta =\delta (\gep) = C (\gep^{\alpha\wedge 1}+ \gep^{1/2}+\eta)\), for some different constant \(C>0\).
Now, we have that there are constants \(C\), \(C_T\) such that the last integral verifies
\begin{equation*}
P_T(I) := \!\!\!\!\! \int\limits_{ \substack{ (r_{j},s_{j})_{j=1}^{\ell} \\ (i_j-1)T < r_j\le s_j\le  i_{j}T } } \!\!\!\!\!  \prod_{j=1}^{\ell} K(r_j-s_{{j-1}})u(s_j-r_j)\dd r_j(\delta_{r_j}(\dd s_j)+ \dd s_j) \leq  C_T \prod_{j=1}^{\ell} \frac{C}{(i_j-i_{j-1})^{1+\frac{\alpha}{2}}} \,.
\end{equation*}
This follows by a standard iteration exactly as in~\cite[Equation (6.5)]{BS11}, combined with Potter's bound \cite[Thm.~1.5.6]{BGT89}. For the iteration, one needs to treat the cases \(i_j-i_{j-1} \geq 2\) and \(i_{j}-i_{j-1}=1\) separately, similarly as in~\cite[Lemma~2.4]{GLT11} (we skip the details). 
Going back to~\eqref{coarsegraining2},  we get that
\begin{equation*}
  \label{copae}
  \bbE\left[ (Z^{Y}_{\beta,nT})^{\theta}\right]\le C_T \sum_{\ell=0}^{n}  \sum_{0< i_1 <\ldots < i_{\ell}\le n }  \prod_{j=1}^{\ell} \frac{C \delta }{(i_j-i_{j-1})^{(1+\frac\alpha2)\theta}} \leq C_T \sum_{\ell =0}^{\infty} \bigg( \sum_{i=1}^{\infty} \frac{C \delta}{(i)^{\theta(1+\frac{\alpha}{2})\theta}} \bigg)^{\ell}  \,,
\end{equation*}
where for the last inequality we have simply dropped the restriction on $i_{\ell}$.
Therefore, if we have fixed~$\theta$ such that $(1+\frac\alpha2)\theta >1$, we may fix $\gep$ small (hence \(\delta\) small) such that 
\[
\sum_{i=1}^{\infty} \frac{C\delta}{i^{(1+\frac\alpha2)\theta}} \le \frac12 \,.
\]
This implies that  $\bbE[ (Z^{Y}_{\beta,nT})^{\theta}] \leq 2 C_T$ for any \(n\geq 1\), which concludes the proof thanks to~\eqref{fracmomo}.
\end{proof}

\begin{proof}[Proof of Proposition~\ref{prop:coarse-graining2}]
Let us assume that $\cA$ in the construction~\eqref{def:g} of $g$ is $\gep$-better.
In this case we use conditional expectation to perform a factorization.
Setting \(\bbP_{[r,s]} (\cdot) := \bbP( \cdot \mid \cF_{\bbR_+ \setminus [r,s]})\), we have 
\begin{equation}\label{copaee}
\bbE\left[ \prod_{j=1}^{\ell} K_w(s_{{j-1}},r_j,Y) G_{i_j}  Z^{Y}_{\beta,[r_{j},s_{j}]}\right]
\!\! =\bbE\left[ \prod_{j=1}^{\ell} K_w(s_{{j-1}},r_j,Y) \bbE_{[r_j,s_j]} \!\!\left[  G_{i_j} Z^{Y}_{\beta,[r_{j},s_{j}]}\right]\right]
\end{equation}
Now, similarly as in~\eqref{replaceterms}, we obtain that
\[
\bbE_{[r_j,s_j]}\left[ G  Z^{Y}_{\beta,[r_j,s_j]} \right]\le C u (r_j-s_j) \bQ_{[r_j,s_j]}\left[ \bbE_{\tau}[ G_{i_j} \mid \cF_{\bbR_+\setminus [r_j,s_j]}]^2 \right]^{1/2} \,,
\]
and the \(\gep\)-better assumption \eqref{singlecellis} implies that 
\begin{equation*}
\bQ_{[r_j,s_j]}\left[ \bbE_{\tau}[ G_{i_j} \mid \cF_{\bbR_+\setminus [r_j,s_j]}]^2 \right]^{1/2} \le \ind_{\{ s_j-r_j \le \gep T\}}+ (\gep+\eta^2)^{1/2} \,.
\end{equation*}
Injecting this back in \eqref{copaee} yields 
\[ 
\bbE\left[ \prod_{j=1}^{\ell} K_w(s_{{j-1}},r_j,Y) G_{i_j}  Z^{Y}_{\beta,[r_{j},s_{j}]}\right]
 \le  (C)^{\ell} \prod_{j=1}^{\ell} K(r_j-s_{{j-1}}) u(s_j-r_j)\left[\ind_{\{ s_j-r_j \le \gep T\}}+ (\gep+\eta^2)^{1/2}\right]\,.
\]
Using the above in \eqref{coarsegraining2}, we can then proceed exactly as in the previous proof: we use \eqref{eq:iterationstep} to get the same bound as in~\eqref{eq:boundPI} and the proof is then identical.
\end{proof}

\subsection{A statement that gathers them all}

In view of Propositions~\ref{prop:coarse-graining}-\ref{prop:coarse-graining2}, the key to our proof is therefore to find some event satisfying~\eqref{singlecell} (or \eqref{singlecellis} in the case of \Cref{prop:largerho}). 
The choice of the event $\cA$ depends on the parameters, and we collect in the following \Cref{prop:key} all the estimates needed to prove \Cref{prop:largerho}, \Cref{rele1} and \Cref{rele2}.
In the case \(\gamma=\frac23\), we need to introduce some more notation to treat the case of a generic slowly varying function \(\varphi(\cdot)\) in \eqref{JPP}.
Define 
\begin{equation}
  \label{defpsi}
\psi(t) := \varphi\left( \frac{1}{K(t)}\right)^3  \int_1^{\frac{1}{K(t)}} \frac{\dd s}{s \gp(s)^3 } \,,
\end{equation}
which is a slowly varying function. Note that it is easy to see that \(\lim_{t\to\infty} \psi(t) =+\infty\), as proven e.g.\ in \cite[Prop. 1.5.9.a]{BGT89}.
Note also that in the case where \(\gp(t) \stackrel{t\to \infty}{\sim} (\log t)^{\kappa}\), we have
\begin{equation}
  \label{psilog}
\psi(t) \sim c_{\kappa}
\begin{cases}
 (\log t) & \quad \text{ if } \kappa < 1/3 \,, \\
 (\log t) (\log \log t) & \quad \text{ if } \kappa = 1/3 \,, \\
 (\log t)^{3\kappa} & \quad \text{ if } \kappa > 1/3 \,.
\end{cases}
\end{equation}

\begin{proposition}
\label{prop:key}
  Assume that~\eqref{JPP} holds, let $\gb\in (\gb_0,2\gb_0)$ and set \(T=\tf(\gb)^{-1}\).
  Then, for any $\gep\in (0,1)$, there exists some \(C_0  = C_0(\gep,J) >0\) such that the following holds, if $\beta$ is sufficiently close to \(\gb_0\) (or equivalently \(T\) sufficiently large):
  \begin{enumerate}
    \item \label{irholarge}
    If \(J(x) \stackrel{|x|\to \infty}{\sim}  |x|^{-(1+\gamma)}\) with $\gamma \in (\frac23,1)$, if \(\rho \geq 1- C_0^{-1} \,T^{-1}\), then there exists an event~\(\cA\) that verifies~\eqref{singlecellis};
    \item \label{ii<2/3}
     If \(J(x) \stackrel{|x|\to \infty}{\sim}  |x|^{-(1+\gamma)}\) with $\gamma \in (0,\frac12)\cup(\frac12,\frac23)$, if $\rho \in (C_0 T^{-\frac{2-\nu}{\nu}},\frac12]$ then there exists an event~\(\cA\) that verifies~\eqref{singlecell}.    
    \item \label{iii2/3}
    If~\eqref{JPP} holds with~$\gamma=\frac23$, if \(\rho  \in ( \frac{C_0}{\psi(T)},\frac12]\) then there exists an event~\(\cA\) that verifies~\eqref{singlecell}.
    \end{enumerate}
\end{proposition}

\noindent
From the above, one concludes easily the proofs of \Cref{prop:largerho} and \Cref{rele1,rele2}.

\begin{proof}[Proof of \Cref{prop:largerho}]
From item~\ref{irholarge} and applying \Cref{prop:coarse-graining2}, for any \(\beta_1 \in (\beta_0,2\beta_0)\) one can find \(\rho\) sufficiently close to \(1\) so that \(\tf(\beta_1,\rho) =0\), \textit{i.e.}\ \(\beta_c(\rho) \geq \beta_1 >\beta_0\).
\end{proof}

\begin{proof}[Proof of \Cref{rele1}]
Define \(\beta_1 := \beta_1(\rho) = \beta_0 + c_1 (\rho/C_0)^{\frac{1}{2-\nu}}\) with \(c\) small enough so that \(\beta_1 \in (\beta_0,2\beta_0)\) and \(\tf(\beta_1) < (\rho/C_0)^{\frac{\nu}{2-\nu}}\), recalling \Cref{homener} (and the fact that \(\alpha\neq \frac12\)).
With this choice we can apply item~\ref{ii<2/3} above with \(T = \tf(\beta_1)^{-1}\), so that \Cref{prop:coarse-graining} shows that \(\tf(\beta_1,\rho) =0\), that is \(\beta_c(\rho) \geq \beta_1 = \beta_0 + c \rho^{\frac{1}{2-\nu}}\), as desired.
\end{proof}

\begin{proof}[Proof of \Cref{rele2}]
Define \(\beta_1 := \beta_1(\rho) = \beta_0 + c_1/\psi^{-1}(C_0/\rho)\) with \(c_1\) small enough so that \(\beta_1 \in (\beta_0,2\beta_0)\) and \(\tf(\beta_1) < 1/\psi^{-1}(C_0/\rho)\), recalling \Cref{homener} (and the fact that \(\nu=2\) in this case).
With this choice we can apply item~\ref{iii2/3} with \(T = \tf(\beta_1)^{-1}\) so that \(\psi(T) \geq \rho/C_0\), and \Cref{prop:coarse-graining} shows that \(\tf(\beta_1,\rho) =0\), that is \(\beta_c(\rho) \geq \beta_1 = \beta_0 + c_1/\psi^{-1}(C_0/\rho) >0\).
The lower bound presented in~\eqref{soluce2} simply corresponds to taking the inverse of \(\psi\) in the case where \(\varphi(t) \sim (\log t)^{\kappa}\), see~\eqref{psilog} above.
\end{proof}

\subsection{A first comment on how to prove that an event is \texorpdfstring{\(\gep\)}{epsilon}-good (or \texorpdfstring{\(\gep\)}{epsilon}-better)}
\label{sec:eventsAB}

Before we prove the three items of \Cref{prop:key}, let us make one comment on how we will prove either \eqref{singlecell} or \eqref{singlecellis}.
While the choice of the event $\cA$ depends highly of the case that we wish to treat, there is indeed a common framework that we will use.
In the same spirit as in~\eqref{eq:boundtildePAT}, we introduce an event \(B\) that may depend on $r$ and $s$ and we observe that 
\begin{equation}\label{popopopop}
  \bQ_{[r,s]}\big( \bbP_{\tau}(\cA^{\cc}) \big) \leq \bQ_{[r,s]}\big[ \bbP_{\tau}(\cA^{\cc}) \ind_{B} \big] + \bQ_{[r,s]}(B^\cc) \,.
\end{equation}
We can thus restrict ourselves to proving that for any \([r,s]\subset [0,T]\) with \(s-r\geq \gep T\), we can find an event $B$ such that 
\begin{equation}
\label{singlecell-B}
 \bbP_{\tau}(\cA^{\cc}) \ind_{B}  \leq \frac{\gep}{2} \,,\quad \text{ and } \quad 
 \bQ_{[r,s]}(B^\cc) \leq \frac{\gep}{2}\,.
\end{equation}
In the case where one needs to prove~\eqref{singlecellis} instead (as in \Cref{prop:key}-\ref{irholarge}), one simply replace \(\bbP_{\tau}(\cA^{\cc})\) by \(\bbP_{\tau}(\cA^{\cc} \mid \cF_{\bbR_+\setminus [r,s]})\).
Recall that in all cases we also need to show that \(\bbP(\cA) \leq \gep\).

\section{Proof of \texorpdfstring{\Cref{prop:key} case \ref{irholarge}}{the proposition case (i)}}
\label{sec:proofI}

We define the events \(\cA\), \(B\) as follows 
\[
\cA = \cA_T:=\Big\{ Y :   \max_{x\in \bbZ}  L_T^Y(x) \ge  (\log T)^2 \Big\}\,,\quad \text{ with } L_T^Y(x) :=\int_0^T \ind_{\{Y_s=x\}} \dd s \,.
\]
For \( [r,s]\subset [0,T]\), we also define $\cA_{[r,s]} \subset \cA$ by 
\[ 
\cA_{[r,s]}:=\Big\{ Y :  \int^{s}_{r} \ind_{\{Y_s=Y_r\}}\dd s \ge  (\log T)^2 \Big\}\,.
\]
Finally, let us define \(B\) as follows (we will use the same event \(B\) in the proof of item~\ref{ii<2/3} of \Cref{prop:key}):
\begin{equation}
\label{def:B1}
B = B_{[r,s]}^{(R)} := \bigg\{   \sum_{i=1}^{|\tau \cap [r,s]|} \ind_{\{\tau_{i}-\tau_{i-1} \in [1,2]\}} \geq R^{-1} T^{\alpha\wedge 1} \bigg\} \,,
\end{equation}
with $R =R(\gep)$ an extra parameter which will be chosen to be large.
Let us recall that in both cases~\ref{irholarge}-\ref{ii<2/3} of \Cref{prop:key}, we have \(J(x)\stackrel{|x|\to \infty}{\sim} |x|^{-(1+\gamma)}\) with \(\gamma \in (0,1)\setminus\{\frac12\}\), so in particular \(K(t)\stackrel{t\to \infty}{ \sim} c_{\gamma} t^{-(1+\alpha)}\) with \(\alpha =\frac{1-\gamma}{\gamma} \in (0,\infty) \setminus \{1\}\). 
In this section, we prove the following three results.

\begin{lemma}
\label{lem:ProbaA}
There is a constant \(c>0\) such that, for any $\rho \in (\frac12,1)$ and any \(T\geq 1\)
\[
\bbP\big(\cA\big)\leq T \exp\big(- c (\log T)^2 \big) \,.
\]
\end{lemma}

\begin{lemma}
  \label{lem1}
For any \(\gep \in (0,1)\), for any \([r,s]\subset [0,T]\) with  \(s-r\geq \gep T\), if $\rho \geq 1- \frac14 \gep T^{-1}$, then for large enough \(T\) we have
\begin{equation*}
   \bbP_{\tau}(\cA^{\cc}_{[r,s]}) \ind_{B}\leq  \frac{\gep}{2} \,.
\end{equation*}
\end{lemma}

Note that by inclusion  we have  $ \bbP_{\tau}(\cA^{\cc} \ | \ \cF_{\bbR\setminus [r,s]}) \le  \bbP_{\tau}(\cA^{\cc}_{[r,s]} \ | \ \cF_{\bbR\setminus [r,s]})=   \bbP_{\tau}(\cA^{\cc}_{[r,s]})$. 
\begin{lemma}
  \label{lem2}
  Assume that \(J(x)\sim |x|^{-(1+\gamma)}\) with \(\gamma \in (0,1)\setminus\{\frac12\}\).
  If $\gep$ is sufficiently small, $R\ge \gep^{-5}$ and~$T$ is sufficiently large then for any \([r,s]\subset [0,T]\) with \(r-s \geq \gep T\),
\[
\bQ_{[r,s]}(B^{\cc}) \leq \gep/2 \,.
\]
\end{lemma}

\noindent
In view of \Cref{sec:eventsAB} (see \eqref{singlecell-B}), this shows that the event $\cA$ satisfies \eqref{singlecellis} for \(T\) sufficiently large.
This proves \Cref{prop:key}-\ref{irholarge}.

\subsection{Proof of \texorpdfstring{\Cref{lem:ProbaA}}{the first lemma}}

Let us note that we have by sub-additivity
\[
\bbP(\cA) \le  \sum_{x\in \bbZ} \bbP\left( L_T^Y(x)> (\log T)^2 \right) \,.
\]
Then, using the strong Markov property at the first time when \(Y_t=x\) and translation invariance, we obtain that \(\bbP( L_T^Y(x)> (\log T)^2 ) \leq \bbP( L_T^Y(x)> 0) \bbP( L_T^Y(0)> (\log T)^2)\), so that bounding \(L_T^Y(0) \leq L_{\infty}^Y(0)\)
\[
\bbP(\cA^{\cc}) \leq  \Big(\sum_{x\in \bbZ^d} \bbP( L_T^Y(x)>0)\Big) \bbP\left( L_{\infty}^Y(0)\ge (\log T)^2 \right) \,.
\]

The first term is simply the expected size of the range of $(Y_s)_{s\in [0,T]}$, which can be bounded from above by the expected number of jumps (including ghost jumps corresponding to $V_i=0$), and is therefore bounded by~$\rho T$.

For the second term, since the walk is transient, $L_{\infty}^Y(0)$ is an exponential random variable with parameter \(\rho p_{\infty}\) with \(p_{\infty}\) is the probability that the discrete-time random walk with transition kernel $J$ never returns to \(0\). 
Indeed, we can write $L_{\infty}^Y(0)=\sum_{i=1}^{G} E_{i}$ where $G$ is a geometric random variable of parameter \(p_{\infty}\) and $E_i$ are independent exponential random variables with parameter $\rho$. 
In particular, we have 
\begin{equation*}
\bbP\big( L_\infty(0)\ge (\log T^2 ) \big) = e^{- p_{\infty} \rho (\log T)^2} \leq e^{- \frac12 p_{\infty} (\log T)^2} \,,
\end{equation*}
recalling that \(\rho \in (\frac12,1)\).
This concludes the proof of \Cref{lem:ProbaA}, with $c=p_{\infty}/2$.
\qed

\subsection{Proof of \texorpdfstring{\Cref{lem1}}{the second lemma}}

We assume to simplify notation that $r=0$.
The idea is that if $(1-\rho)T$ is very small, then under $\bbP_{\tau}$, with large probability $Y$ comes back to zero at every point in~$\tau$ (and this estimate is uniform in the point process $\tau\subset [0,s]$).
Indeed using the representation of \Cref{lem:sizebiased} for~\(\bbP_{\tau}\), we get that 
\[
\bbP_{\tau} \big(\forall t\in \tau, \ Y_{t}=0 \big) = \prod_{i=1}^{|\tau|} \P(W_{\rho (\tau_i-\tau_{i-1}) } =0 \mid W_{\tau_i-\tau_{i-1}} =0) \,.
\]
Then, using Markov's property and then \Cref{lem:unimod}, we get that 
\[
  \P(W_{\rho t } =0 \mid W_{t} =0)  = \frac{ \P(W_{\rho t } =0) \P(W_{(1-\rho)t}=0)  }{ \P(W_{t} =0)} \geq \P(W_{(1-\rho)t}=0) \,.
\]
Additionally, we have that $\P(W_{(1-\rho)t}=0)\ge e^{-(1-\rho)t}$, since $e^{-(1-\rho)t}$ is the probability of having no jump at all.
All together, we have that for \(\tau\subset [0,s]\) with \(\tau_0=0\) and \(s\in \tau\),
\begin{equation}
  \label{pinned}
   \bbP_{\tau} \big(\forall t \in \tau, \ Y_{t}=0 \big) \ge  \prod_{i=1}^{|\tau|}  e^{-(1-\rho) (\tau_i-\tau_{i-1})} =  e^{- (1-\rho) s} \,.
\end{equation}
Since on the event \(B\) the number of renewal points is of order $T^{\alpha \wedge 1}$ (recall~\eqref{def:B1}), this in turns will imply that the total time spent at zero is typically be much larger than $(\log T)^2$.
More precisely, write
\[
\bbP_{\tau}(\cA^{\cc}_{[0,s]}) \leq  \bbP_{\tau}\big(L_s(0) \leq (\log T)^2 \ ; \  \forall t\in \tau, Y_{t}=0  \big) +  \bbP_{\tau}\big(\exists t\in \tau, \, Y_t \neq 0\big) \,.
\]
Thanks to~\eqref{pinned}, the second term is smaller than $1-e^{-(1-\rho)s} \leq (1-\rho)T \leq \frac{\gep}{4}$, recalling the condition on \(\rho\).
On the other hand, the first term is smaller than 
\begin{equation}
  \label{ajaj}
   \hat \bbP_{\tau}\big( L_s(0) \leq (\log T)^2\big)  \,, \quad \text{ with } \hat\bbP_{\tau} \big( \cdot \big)  :=  \bbP_{\tau}\big( \, \cdot \mid  \forall t\in \tau,\  Y_{t}=0 \big) \,.
\end{equation}
We let $i_i\dots,i_{N}$ denote the ordered enumeration of the set $\{ i  : \tau_{i}-\tau_{i-1}\in (1,2] ; \tau_i\le s \}$.
Then, for $k\le N$ let us set \(\chi_k\) the indicator of the event~$\{\forall s\in [\tau_{i_k},\tau_{i_{k+1}}], Y_s= 0  \}$.
Thanks to \Cref{lem:sizebiased}, the variables $(\chi_{k})_{1\leq k \leq N}$ are independent Bernoulli variables under $\hat \bbP_{\tau}$, with parameter
\[
\frac{\P\big(\forall u\in [0,\rho(\tau_{i_k}-\tau_{i_k-1})],\, W_{u}=0  \mid W_{\tau_{i_k}-\tau_{i_k-1}} =0 \big)  }{\P\big( W_{\rho(\tau_{i_k}-\tau_{i_k-1})} =0 \mid W_{\tau_{i_k}-\tau_{i_k-1}} =0 \big) } = \frac{\P\big(\forall u\in [0,\rho(\tau_{i_k}-\tau_{i_k-1})],\, W_{u}=0 \big)  }{\P( W_{\rho(\tau_{i_k}-\tau_{i_k-1})} =0) }\,.
\] 
Therefore, bounding the denominator by~\(1\) and then using the fact that $\rho(\tau_{i_k}-\tau_{i_k-1})\le 2$, we get that the parameter verifies
\(
\hat\bbP_{\tau}(\chi_k=1) \geq e^{-2} .
\)
Then, on the event that $N\ge R^{-1}T^{1\wedge \alpha}$ (\textit{i.e.}\ on the event \(B\)), we have
 \begin{equation*}
     \hat \bbP_{\tau}\big( L_s(0) \leq (\log T)^2\big)\le    \hat \bbP_{\tau}\bigg( \sum_{k=1}^{R^{-1}T^{1\wedge \alpha}} \chi_k \leq (\log T)^2\bigg)\le \exp\left( -c_R T^{1\wedge \alpha}\right) \,,
 \end{equation*}
where the last inequality is a simple consequence of a large deviation estimate (using for instance Hoeffding's inequality), provided that \((\log T)^2 \leq \frac12 e^{-2} R^{-1} T^{1\wedge \alpha}\).
This concludes the proof of \Cref{lem1} if \(T\) is large enough.
\qed

\subsection{Proof of \texorpdfstring{\Cref{lem2}}{the third lemma}}
Assume again to simplify notation that $r=0$.
First of all, \(\bQ_{[0,s]}(B^{\cc} )\) is bounded by
\[
\bQ_{[0,s]}\bigg( \sum_{i=1}^{|\tau \cap [0,s/2]|} \ind_{\{\tau_{i}-\tau_{i-1} \in (1,2]\}} < R^{-1}T^{\alpha \wedge 1} \bigg)
\leq C  \bQ\bigg( \sum_{i=1}^{|\tau \cap [0,s/2]|} \ind_{\{\tau_{i}-\tau_{i-1} \in (1,2]\}} <R^{-1}T^{\alpha\wedge 1}\bigg) \,,
\]
where we have used \Cref{lem:removecond} to remove the conditioning \(s\in \tau\), at the cost of a multiplicative constant \(C>0\).
Then, omitting integer parts to lighten notation, we have that the last probability is bounded by
\begin{equation*}
  \bQ\bigg( \sum_{i=1}^{R^{-1/2}T^{\alpha\wedge 1}} \ind_{\{\tau_{i}-\tau_{i-1} \in (1,2]\}} < R^{-1}T^{\alpha\wedge 1}\bigg)+ \bQ\left( |\tau \cap [0,s/2]| < R^{-1/2} T^{\alpha\wedge 1}\right) \,.
\end{equation*}
Using that the \(\ind_{\{\tau_{i}-\tau_{i-1} \in (1,2]\}}\) are i.i.d.\ Bernoulli random variables with parameter \(\int^2_1 K(s)\dd s\), we get by a large deviation estimate (e.g.\ Hoeffding's inequality) that the first term decays like $e^{-c T^{1\wedge \alpha}}$, provided that $R$ is large enough so that \(R^{-1/2} < \frac12 \int^2_1 K(s)\dd s\).
For the second term, using the assumption that $s\ge \gep T$, we simply write
\[
\bQ\big( |\tau \cap [0,s/2]| < R^{-1/2} T^{\alpha\wedge 1} \big) \leq \bQ\big( \tau_{R^{-1/2} T^{\alpha\wedge 1}}> \gep T/2 \big) \,.
\]
Then, using Markov's inequality, we have 
\begin{equation}\label{decqwer}
 \bQ(\tau_{k}> A)\le \bQ\bigg(\sum_{i=1}^k [(\tau_i-\tau_{i-1})\wedge A] \ge A\bigg)\le  \frac{k}{A} \bQ(\tau_1\wedge A) .
\end{equation}
Applying this with \(k= R^{-1/2} T^{\alpha\wedge 1}\) and \(A= \gep T/2\), we thus get that 
\begin{equation*}
  \bQ\big(\tau_{R^{-1/2} T^{\alpha\wedge 1}}> \gep T/2 \big) \le  \frac{2 T^{\alpha\wedge 1-1}}{\gep R^{1/2}} \int^{\gep T/2}_0 \bar K(t)\dd t \le \frac{C}{R^{1/2} \gep^{\alpha\wedge 1}}  
\end{equation*}
where in the last inequality is valid for all $T\ge 1$, for a constant $C$ which depends only on the particular expression for $J(\cdot)$ (recall that \(\bar K(t) \sim \bar c_{J} t^{-\alpha}\)  with \(\alpha \in (0,\infty)\setminus\{1\}\)). 
Since \( \gep^{-\alpha\wedge 1}\leq \gep^{-1}\) and \(R \geq \gep^{-5}\), this shows that \(\bQ_{[r,s]}(B^{\cc}) \leq C \gep^{3/2}\), which concludes the proof if \(\gep\) has been taken small enough.
\qed

\section{Proof of Proposition \ref{prop:key}, case \ref{ii<2/3}}
\label{sec<2/3}

% Recall that here we assume that \(J(x)\sim c_{\gamma} |x|^{-(1+\gamma)}\) with \(\gamma \in (0,\frac23)\). 
Again, let us now define the event \(\cA\); the event \(B\) is still defined as in~\eqref{def:B1}.
For an interval $I\subset [0,T]$, let us define $\cJ_{I} = |\{i : \vartheta_i \in I\} |$ the number of jumps of $Y$ in the time interval $I$ (recall the Poisson construction of \Cref{sec:propertiesRW}).
We then introduce the event
\begin{equation}
 \cA:=\left\{ \cJ_{(0,T]}<  \rho T- R \sqrt{\rho T} \right\} \,,
\end{equation}
where the constant $R$ will be chosen sufficiently large later on.
Since under~$\bbP$ the number of jumps $\cJ_{(0,T]}$ is a Poisson random variable with parameter~$\rho T$, a simple application of Chebyshev's inequality shows that $\bbP(\cA)\leq R^{-2}\leq \gep$ provided that $R \geq \gep^{-1/2}$.
Hence, the first part of~\eqref{singlecell} holds. 
To prove that \eqref{singlecell-B} holds, we rely on Lemma \ref{lem2} to control \(\bQ_{[r,s]}(B^{\cc})\) and on the following lemma.

\begin{lemma}
  \label{lem:A2}
  For any \(\gep \in (0,1)\) there exist \(R = R(\gep)\) and \(C_0= C_0(\gep,R,J)\) such that the following holds. 
  For any \([r,s]\subset [0,T]\) with \(s-r\geq \gep T\), if \(\rho \geq C_0 T^{1-2(\alpha \wedge 1)}\), then for large enough~\(T\) we have 
  \[
  \bbP_{\tau}(\cA^{\cc}) \ind_B \leq \frac{\gep}{2} \,.
  \]
\end{lemma}

\noindent
This concludes the proof of \Cref{prop:key}-\ref{ii<2/3}, since we have \(\nu = \frac{1}{\alpha \wedge 1}\), so that \(1-2(\alpha\wedge 1)=  -\frac{2-\nu}{\nu}\).
\qed

\begin{proof}[Proof of \Cref{lem:A2}]
For $\tau=\{\tau_0,\tau_1,\dots,\tau_m\} \subset [r,s]$ fixed and \(\tau_0=r\), \(\tau_m=s\), we can decompose the number of jumps as follows:
\begin{equation*}
 \cJ_{(0,T]}:= \cJ_{(0,r]}+ \sum_{k=1}^{m}   \cJ_{(\tau_{i-1},\tau_i]} + \cJ_{(s,T]} \,.
\end{equation*}
We then split the contribution into two parts: \(\cJ_{(0,T]} = \cJ_1 +\cJ_2\) where \(cJ_1\) contains the terms that are ``most affected'' by changing the measure from \(\bbP\) to \(\bbP_{\tau}\).
More precisely, we set
\begin{equation*}
  \begin{split}
 \cJ_1&:= \sum_{k=1}^{m}   \cJ_{(\tau_{i-1},\tau_i]} \ind_{\{\tau_i-\tau_{i-1}\in (1,2]\}} ,\\
\cJ_2&:= \cJ_{(0,T]} - \cJ_1 =\cJ_{(0,r]}+ \sum_{k=1}^{m}   \cJ_{(\tau_{i-1},\tau_i]} \ind_{\{\tau_i-\tau_{i-1}\notin (1,2]\}} + \cJ_{(s,T]} .
 \end{split}
\end{equation*}
We then have that \(\cA^{\cc} \subset \cA_1 \cup \cA_2\) an so \(\bbP_{\tau}(\cA^{\cc}) \leq \bbP_{\tau}(\cA_1)+\bbP_{\tau}(\cA_2)\), with
\[
\cA_1 := \left\{ \cJ_{1}\ge  \bbE[\cJ_1] - 2R \sqrt{\rho T} \right\} \quad  \text{ and } \quad \cA_2 := \left\{ \cJ_{2}\le  \bbE[\cJ_2] +R \sqrt{\rho T} \right\}\, ,
\]
where we have also used that \(\bbE[\cJ_1]+\bbE[\cJ_2] = \rho T\).
First of all, using the comparison property of \Cref{lem:domination}, we get that \(\bbP_{\tau}(\cJ_2 \geq t) \leq \bbP(\cJ_2\geq t)\) for any \(t\geq 0\). 
Therefore, \(\bbP_{\tau}(\cA_2)\leq \bbP(\cA_2)\), so that using Chebyshev's inequality and the fact that \(\Var_{\bbP}(\cJ_2) \le \rho T\), we get that
\[
\bbP_{\tau}(\cA_2) \leq \bbP\big( \cJ_{2}\ge  \bbE[\cJ_2] + R \sqrt{\rho T} \big) \leq \frac{\Var_{\bbP}(\cJ_2)}{R^2 \rho T} \leq \frac{1}{R^2} \le \frac{\gep}{4}
\]
the last inequality holding for \(R \geq \gep^{-1}\) with \(\gep\leq \frac14\).
To estimate \(\bbP_{\tau}(\cA_1)\), we need to prove the following.
\begin{claim}
\label{espetvar}
 We have 
 \(
 \bbE[\cJ_1]-\bbE_{\tau}[\cJ_1]\ge c\rho \sum\limits_{i=1}^{m} \ind_{\{\tau_i-\tau_{i-1}\notin (1,2]\}}
 \)
and \( \Var_{\bbP_{\tau}}[\cJ_2]\le 3 \rho T\).
\end{claim}

With \Cref{espetvar} at hand, on the event $B$  we  have 
\( \bbE[\cJ_2]-\bbE_{\tau}[\cJ_2] \geq c \rho R^{-1} T^{1\wedge \alpha}\ge 3R \sqrt{\rho T} \), where the second inequality holds if $\rho \geq C_0 T^{1-2(\alpha\wedge 1)}$ and $C_0$ is sufficiently large.
Therefore,  we obtain that on the event~\(B\)
\[
\bbP_{\tau}(\cA_2) \leq \bbP_{\tau}\Big( \cJ_{2}\le  \bbE_{\tau}[\cJ_2] - R \sqrt{\rho T} \Big) \leq \frac{\Var_{\bbP_{\tau}}[\cJ_2]}{R^2 \rho T}\le 
 \frac{3}{R^2}  \leq \frac{\gep}{4}\,,
\]
using Chebyshev's inequality, then \Cref{espetvar} and taking ~\(R\geq \gep^{-1}\) with \(\gep \leq \frac{1}{12}\) for the last inequality.
This concludes the proof of \Cref{lem:A2}.
\end{proof}

\begin{proof}[Proof of \Cref{espetvar}]
Note that since the number of jump is independent on each interval $(\tau_{i-1},\tau_i]$ it is sufficient to make the computation for one interval and then sum it.
Using the stochastic comparison of~\Cref{lem:domination}, we get that \(\bbE[f(\cJ_{(\tau_i-\tau_{i-1}]})] \geq \bbE_{\tau}[f(\cJ_{(\tau_i-\tau_{i-1}]})]\) for any non-decreasing function \(f:\bbR_+\to \bbR_+\).
Therefore, using the identity \(\bbE[\cJ] =\bbE[\cJ\vee 1] - \bbP(\cJ=0)\), and since \(x\mapsto x\vee 1\) is non-decreasing, we get that 
\begin{equation*}\begin{split}
\bbE[\cJ_{(\tau_i-\tau_{i-1}]}]- \bbE_{\tau}[\cJ_{(\tau_i-\tau_{i-1}]}]
&\geq \bbP_{\tau}\big[\cJ_{(\tau_i-\tau_{i-1}]}=0\big]- \bbP\big[\cJ_{(\tau_i-\tau_{i-1}]}=0 \big]\\
&= \frac{e^{-\rho(\tau_i-\tau_{i-1}) } \P(W_{(1-\rho)(\tau_i-\tau_{i-1})}=0 )}{\P(W_{(\tau_i-\tau_{i-1})}=0 )} - e^{-\rho(\tau_i-\tau_{i-1}) }  \,.
\end{split}\end{equation*}
Then, using the fact that \(K(t) = \beta_0 \P(W_t=0)\) is Lipschitz on \([0,2]\), we get that \(\frac{K( (1-\rho)t)}{K(t)} -1\geq c \rho t \) for any \(t\in (1,2]\) and \(\rho \in (0,\frac12)\).
Therefore, for \(\tau_i-\tau_{i-1}\in (1,2]\) and \(\rho \in (0,\frac12)\), we get that 
\[
\bbE[\cJ_{(\tau_i-\tau_{i-1}]}]- \bbE_{\tau}[\cJ_{(\tau_i-\tau_{i-1}]}] \geq  e^{-2\rho} c\rho (\tau_i-\tau_{i-1}) \geq c' \rho \,.
\]
This concludes the lower bound on \(\bbE[\cJ_2]- \bbE_{\tau}[\cJ_2]\).
For the variance we simply observe that, using again  \Cref{lem:domination},
\begin{equation*}
 \Var_{\bbP_{\tau}}[\cJ_{(\tau_i-\tau_{i-1}]}]\le \bbE_{\tau}\left[ (\cJ_{(\tau_i-\tau_{i-1}]})^2\right] 
 \le  \bbE\left[ (\cJ_{(\tau_i-\tau_{i-1}]})^2\right] \,.
\end{equation*}
Now, the right-hand side is equal to \(\rho(\tau_i-\tau_{i-1}) +  \rho^2(\tau_i-\tau_{i-1})^2 \leq 3 \rho(\tau_i-\tau_{i-1})\), since \(\tau_i-\tau_{i-1} \in (1,2]\). 
Summing over \(i\) we obtain that \(\Var_{\bbP_{\tau}}[\cJ_2]\le 3 \rho T\).
\end{proof}

\section{Proof of \texorpdfstring{\Cref{prop:key} case~\ref{iii2/3}}{the proposition case (iii)}}
\label{sec:proofII}

\subsection{Organisation and decomposition of the proof}

As above, we first  introduce the events \(\cA\) and \(B\) which we will use in~\eqref{singlecell-B}.
Similarly to the previous section, we consider an event of the form
\begin{equation}
  \label{def:Amarginal}
\cA = \left\{ F_T - \bbE[F_T] \leq - \gep^{-1} \,\sqrt{\bbVar(F_T)} \right\}
\end{equation}
for some $\cF_T$-measurable random variable $F_T$ and some \(R\) large (in fact, \(R=\gep^{-1}\) is enough for our purpose). 
Thanks to Chebyshev's inequality, it is clear that $\bbP(\cA)\leq \gep^{2}\le  \gep$.
% Rather than taking~$F_T$ as a function of $Y$ we make $F_T$ a function of the Poisson process used to construct \(Y\).
Similarly to what was done in \Cref{sec<2/3}, we use a functional $F_{T}$ that counts the number of jumps, but we also weight them by a coefficient which depends on their amplitude.
Recalling the Poisson construction in Section~\ref{sec:propertiesRW}, for an interval \(I \subset [0,T]\), define
\begin{equation}
  \label{def:F}
 F_{T}:= \sum_{i: \vartheta_i\in (0,T]} \xi(U_i) \ind_{\{ U_i K(T) \le 2 \beta_0\}} \quad  \text{ where } \quad  \xi(k):=  k^{1/3} \gp(k)^{-2} \,.
\end{equation}
The new measure $\bbP_{\tau}$ has the effect to make the walk~$Y$ jump less frequently (recall \Cref{lem:sizebiased} and \Cref{lem:domination}), so that $\bbE_{\tau}[F_T]\le \bbE[F_T]$. 
It affects jumps of different size in a different way and our specific choice of \(\xi(\cdot)\) is designed to make to make the renormalized shift of the expectation $\frac{\bbE[F_T]- \bbE_{\tau}[F_T]}{\sqrt{\bbVar(F_T)}}$ as large as possible.
Since $\bbP(\cA)\leq \gep$ almost by definition, we only need to find an event $B$ such that \eqref{singlecell-B} holds.
The event $B$ needs to ensure that the expectation shift $\bbE[F_T]- \bbE_{\tau}[F_T]$ is typical. 
We set
\begin{equation}\label{defbst}
B = B_{[r,s]}^{(\delta)}:= \bigg\{ \sum_{j=1}^{|\tau\cap[r,s]|} \xi\bigg( \frac 1 {K(\Delta \tau_j)} \bigg) \geq  \delta \frac{S(T)}{ TK(T)} \bigg\} \quad \text{ with }  S(T):=\int^{\frac 1 {K(T)}}_1 \frac{\dd s}{s \varphi(s)^3} \,,
\end{equation}
where $\Delta \tau_i:=\tau_{i}-\tau_{i-1}$ and $\delta=\delta(\gep)=\gep^5$.
A first requirement for the proof is to show that $B$ is typical.
\begin{lemma}
  \label{lem:B3}
  There exists $\gep_0$ such that for any \(\gep\in (0,\gep_0)\), setting $\delta=\gep^5$, for any $T$ sufficiently large and \([r,s]\subset [0,T]\) with \(s-r\geq \gep T\), we have 
  \(
  \bQ_{[r,s]} (B^{\cc}) \leq \gep/2 \,.
  \)
\end{lemma}

To estimate \(\bbP_{\tau}(\cA^{\cc})\) and conclude the proof of \eqref{singlecell-B}, we decompose $F_{T} = F_1+F_2$ into two parts~$F_1$ and~$F_2$, where $F_1$ is the sum containing the terms which are ``most affected'' by changing the measure from $\bbP$ to $\bbP_{\tau}$.
For a given realization of \(\tau = (\tau_j)_{j=0}^m\) with \(\tau_0=r\), \(\tau_m=s\), we define
\begin{equation}
  \label{def:F1F2}
 F_1:=\sum_{j=1}^m H_{j}   \qquad
  \text{ with } \quad   H_j := \sum_{i: \vartheta_i \in (\tau_{j-1},\tau_j]}   \xi(U_i)   \ind_{\{\beta_0 \leq U_i  K(\Delta \tau_j) \le 2 \beta_0\}} \,,
\end{equation}
and \(F_2:= F_{T}- F_1\).
Then, we set
\[
\cA_1:= \left\{ F_1\ge \bbE[ F_1] - 2 \gep^{-1} \sqrt{\Var_{\bbP}(F_T)} \right\}   \  \text{ and } \  
  \cA_2:= \left\{ F_2\ge \bbE[ F_2 ] + \gep^{-1} \sqrt{\Var_{\bbP}(F_T)} \right\} \,,
\]
so that $\cA^{\cc}\subset \cA_1\cup \cA_2$ and in particular
\(\bbP_{\tau}(A^{\cc})\le \bbP_{\tau}(\cA_1)+ \bbP_{\tau}(\cA_2)\).
We then need the following estimates on the variances of $F_T, F_1$ and of the expectation shift $\bbE[F_1]-\bbE_{\tau}[F_1]$.
Recall that \(S(T)\) is defined in~\eqref{defbst}.
\begin{lemma}
  \label{lem:var1}
There are constants $c,C>0$ such that 
\begin{equation}\label{frwrt}
c  \rho T S(T)\le  \Var_{\bbP}[F_T]  \le C \rho T S(T)\,.
\end{equation}
Additionally, if \(\rho\) is small enough we have \(\Var_{\bbP_{\tau}}[F_1] \leq 2 \Var_{\bbP}[F_T]\) \,.
\end{lemma}

\begin{lemma}
  \label{lem:shiftE}
  There is a constant \(c>0\) such that, for any \(\rho \in (0,\frac12)\), we have
  \[
  \bbE[F_1] - \bbE_{\tau}[F_1] \geq c \rho  \sum_{j=1}^{|\tau|} \xi\bigg(  \frac 1 {K(\Delta(\tau_j))}\bigg) \,.
  \]
  In particular, on the event \(B\) we have
  \begin{equation}\label{frwrt2}
   \bbE[F_1] - \bbE_{\tau}[F_1] \geq \frac{c \rho \delta S(T)}{TK(T)}.
  \end{equation}
\end{lemma}

\begin{rem}
The fact that the quantity $S(T)$ appears both in the expression of the variance of $F_T$ in~\eqref{frwrt} and in that of the typical value for the expectation shift \(\bbE[F_1] - \bbE_{\tau}[F_1]\) is not a coincidence, but is a consequence of the choice we have made for $\xi$. Having the same integral expression appearing in both computation turns out to be optimal for our purpose.
\end{rem}

\begin{proof}[Conclusion of the proof of \Cref{prop:key}-\ref{iii2/3}] Let us first observe that we can without loss of generality assume that $\gep$ is as small as desired.
Thanks to the fact that \(\bbP(\cA)\leq \gep\) and because of \Cref{lem:B3}, we only need to prove the first part of~\eqref{singlecell-B}.
Using that \(\bbP_{\tau}(A^{\cc})\le \bbP_{\tau}(\cA_1)+ \bbP_{\tau}(\cA_2)\) we therefore need to show that 
\begin{equation}
\bbP_{\tau}(\cA_1)\ind_B\le \frac{\gep}{4}  \quad  \text{ and }  \quad \bbP_{\tau}(\cA_2)\le  \frac{\gep}{4} \,.
\end{equation}

Using the stochastic comparison of \Cref{lem:domination} and since \(\xi\) is a non-negative function (so~\(F_2\) is a non-decreasing function of \(\bar \cU\)), we have that \(\bbP_{\tau}(\cA_2) \leq \bbP(\cA_2)\).
Applying Chebyshev's inequality and then using that \(\Var_{\bbP}\left(F_2\right)\leq \Var_{\bbP}\left(F_T\right)\), we therefore get that (assuming that $\gep\le 1/4$)
\[
\bbP_{\tau}(\cA_2)\le \bbP(\cA_2)\le \frac{\gep^2\Var_{\bbP}\left(F_2\right)}{ \Var_{\bbP}\left(F_T\right)}\le  \gep^2 \le \gep/4\,.
\]

Turning now to $\cA_1$, combining \eqref{frwrt2} and \eqref{frwrt}, we have on the event \(B\) that
\begin{equation*}
 \frac{\big(\bbE_{\tau}[F_1] - \bbE[F_1]\big)^2}{\Var_{\bbP}(F_T)} \ge \frac{c^2  \delta^2}{C}\frac{\rho  S(T)}{ T^{3}K(T)^2 } \ge c'\gep^{10} \rho \psi(T) \,.
\end{equation*}
To obtain the last inequality above, we used the fact that $T^{-3} K(T)^{-2} $ is of the same order as $ \gp(1/K(T))$ thanks to~\eqref{implicit} (recall here that \(\gamma=\frac23\)), together with the definition \eqref{defpsi} of \(\psi\), which can be rewritten as \(\psi(T) = \gp(1/K(T))^{3} S(T)\). 
Using now the assumption in item~\ref{iii2/3} of \Cref{prop:key}, we get that this is bounded from below by \(c' \gep^{10} C_0\).
Hence, taking $C_0$ sufficiently large (how large depends on \(\gep\)), we have that $\bbE[F_1] - \bbE_{\tau}[F_1]\ge 3\gep^{-1} \sqrt{\Var_{\bbP}(F_T)}$ on the event \(B\).
Hence, using Chebyshev's inequality and then the second part of \Cref{lem:var1}, we have (for $\gep\le 1/8$)
\[
\bbP_{\tau}(\cA_1) \leq \bbP_{\tau}\left( F_1-\bbE_{\tau}[F_1]  \geq  \gep^{-1} \sqrt{\Var_{\bbP}[F_T]} \right) \le \frac{\gep^2\Var_{\bbP_{\tau}}[F_1]}{\Var_{\bbP}[F_T]} \leq 2\gep^4 \le \gep/4 .
\]
\end{proof}
%
% \Cref{lem:B3} also shows that \(\bQ_{[r,s]}(B^\cc) \leq \frac{\gep}{2}\), so that~\eqref{singlecell-B} holds, which gives the desired conclusion.
% It remains to prove the above lemmas.

% We now prove \Cref{lem:B3,lem:var1,lem:shiftE}.

\subsection{Proof of \texorpdfstring{\Cref{lem:B3}}{the lemma}}

As above, let us set $r=0$ to simplify notation. 
Also, as in the proof oc \Cref{lem2}, considering only the sum up to time $s/2$ and removing the conditioning thanks to \Cref{lem:removecond} at the cost of a multiplicative constant \(C\).
We are left with showing that
\[
\bQ \bigg(  \sum_{j=1}^{|\tau\cap [0,s/2]|} \xi\Big( \frac{1}{K(\Delta\tau_j)}\Big) \leq   \frac{\delta S(T)}{TK(T)} \bigg)  \le  \frac{\gep}{2C} \,.
\]
Now, using the assumption $s\le \gep T$, the above is bounded by
\begin{equation}
  \label{twoterms}
\bQ\bigg( |\tau\cap [0,\gep T/2]| \leq  \frac{\sqrt{\delta}}{T K(T)} \bigg)
+ \bQ \bigg(  \sum_{j=1}^{\frac{\sqrt{\delta}}{TK(T)}} \xi\Big( \frac{1}{K(\Delta\tau_j)}\Big)\ind_{\{\Delta \tau_j\le T\}}  \leq  \delta \frac{S(T)}{TK(T)} \bigg) \,.
\end{equation}
We need to bound each term by $\gep/4C$.
For the first term, we use the truncated Markov inequality~\eqref{decqwer} with $A=\gep T/2$ and $k= \sqrt{\delta}/ T K(T)$ so we obtain that for $T\ge T_0(\gep)$ sufficiently large
\begin{equation*}
 \bQ\Big( |\tau\cap [0,\gep T/2]| \leq  \frac{\sqrt{\delta}}{T K(T)} \Big)\le   \frac{\sqrt{\delta} \bQ(\tau_1 \wedge (\gep T/2))}{TK(T)}\le 
 5 \frac{\sqrt{\delta} \gep  K(\gep T/2) }{K(T)}  \le 40 \sqrt{\delta/\gep}\le \gep/4C,
\end{equation*}
To obtain the second and third inequalities, we use the fact  since \(K(\cdot)\) is regularly varying with exponent~\(-\frac 3 2\), so that  \(\bQ(\tau_1\wedge A)\stackrel{A\to \infty}{\sim} 4 T^2 K(T)\) and $K(aT)/K(T)\stackrel{T\to\infty}\sim a^{-3/2}$.
% 
%  (recall that ).
% Using Potter's bound~\cite[Thm.~1.5.6]{BGT89}, we get that this is bounded by a constant times \(\sqrt{\delta}\gep^{-1}\leq \gep^2\).
In the last inequality we used $\delta=\gep^5$ and assumed that $\gep\ge 1/(160C)$.
To estimate the second term in~\eqref{twoterms}, we need to estimate the mean and variance of the i.i.d.\ variables appearing in the sum.
We are going to prove  that the two following estimates are satisfied (for some  \(c>0\)),
\begin{align}
  \label{1stand2nd}
\bQ\bigg[ \xi\Big( \frac{1}{K(\tau_1)}\Big) \ind_{\{\tau_1 \leq T\}} \bigg] \geq c   S(T),\\
  \label{1stand2nd-bis}
  \lim_{T\to \infty}\frac{T K(T)}{S(T)^2} \bQ\bigg[ \xi\Big( \frac{1}{K(\tau_1)}\Big)^2 \ind_{\{\tau_1 \leq T\}} \bigg] =0.
\end{align}
The first identity \eqref{1stand2nd} guarantees that for $T$ sufficiently large
$$ \bQ \bigg(  \sum_{j=1}^{\frac{\sqrt{\delta}}{TK(T)}} \xi\Big( \frac{1}{K(\Delta\tau_j)}\Big)\ind_{\{\Delta \tau_j\le T\}}  \bigg)\ge \frac{c \sqrt{\delta} S(T)}{TK(T)}\ge \frac{2\delta S(T)}{T K(T)}   $$
provided that \(\delta\) is small enough. 
Hence, by Chebyshev's inequality we have for $T$ sufficiently large
\[
\bQ \bigg(  \sum_{j=1}^{\frac{\sqrt{\delta}}{TK(T)}} \xi\Big( \frac{1}{K(\Delta\tau_j)}\Big)\ind_{\{\Delta \tau_j\le T\}}  \leq  \delta \frac{S(T)}{TK(T)} \bigg)\le  \delta^{-3/2} \frac{T K(T)}{S(T)^2} \bQ\bigg[ \xi\Big( \frac{1}{K(\tau_1)}\Big)^2 \ind_{\{\tau_1 \leq T\}} \bigg]\le \frac{\gep}{4C},
\]
where the last inequality is a consequence of ~\eqref{1stand2nd-bis}.
We are left with the proof of ~\eqref{1stand2nd}-\eqref{1stand2nd-bis}. 
For the mean~\eqref{1stand2nd}, we have
\[
\bQ\bigg[ \xi\Big( \frac{1}{K(\tau_1)}\Big) \ind_{\{\tau_1 \leq T\}} \bigg] = \int_{0}^T K(s)  \xi\Big( \frac{1}{K(s)}\Big)\dd s =\int_{0}^{\frac{1}{K(T)}} \frac{\xi(u)}{u^3|K'\left( K^{-1}(1/u) \right)| } \dd u 
\]
by a simple change of variable.
Now, \cite[Proposition 3.2]{BLirrel} shows that $|K'(s)|\stackrel{s\to \infty}{\sim} \frac{3}{2} s^{-1}K(s)$
and \eqref{implicit} gives that \(K^{-1}(1/u) \stackrel{u\to \infty}{\sim} c_{2/3}  u^{2/3}/\gp(u)\).
Recalling that \(\xi(u)=u^{1/3} \varphi(u)^{-2}\), we therefore get that
\begin{equation}
  \label{asympxiu}
 \frac{\xi(u)}{u^3|K'\left( K^{-1}(1/u) \right)| } \stackrel{u\to \infty}{\sim}  \frac{ 2 K^{-1}(1/u)}{3 u^{5/3} \varphi(u)^{2}} \stackrel{u\to \infty}{\sim} \frac{2 c_{2/3}}{ 3 u \varphi(u)^3} \,.
\end{equation}
Therefore, if the integral \(S(T) = \int_{1}^{1/K(T)} \frac{\dd u}{u \gp(u)^3} \dd u \) diverges, the above shows that the mean is asymptotically equivalent to \(\frac{2}{3} c_{2/3} S(T)\) and \eqref{1stand2nd} holds.
If on the other hand the integral converges, we also get that the mean is convergent and \eqref{1stand2nd} also holds.
For the second moment~\eqref{1stand2nd-bis}, with the same type of computation as for the mean, we obtain
\begin{equation*}
  \bQ\bigg[ \xi\Big( \frac{1}{K(\tau_1)}\Big)^2 \ind_{\{\tau_1 \leq T\}} \bigg] = \int_{0}^T K(s)  \xi\Big( \frac{1}{K(s)}\Big)^2\dd s =\int_{0}^{\frac{1}{K(T)}} \frac{\xi(u)^2}{u^3 |K'\left( K^{-1}(1/u) \right)| } \dd u \,.
\end{equation*}
Then, similarly to~\eqref{asympxiu}, we get 
\[
 \frac{\xi(u)^2}{u^3|K'\left( K^{-1}(1/u) \right)| } \stackrel{u\to \infty}{\sim}  \frac{ 2 K^{-1}(1/u)}{3 u^{4/3} \varphi(u)^{4}} \stackrel{u\to \infty}{\sim} \frac{2 c_{2/3}}{ 3 u^{2/3} \varphi(u)^5} \,.
\]
Since the integral \(\int_1^{s} \frac{\dd u}{u^{2/3} \varphi(u)^5}\) diverges, we get that
\[
\bQ\bigg[ \xi\Big( \frac{1}{K(\tau_1)}\Big)^2 \ind_{\{\tau_1 \leq T\}} \bigg]  \stackrel{T\to \infty}{\sim} \frac{2 c_{2/3}}{3} \int_1^{\frac{1}{K(T)}} \frac{1}{u^{2/3} \varphi(u)^5}  \stackrel{T\to \infty}{\sim} \frac{2 c_{2/3}}{K(T)^{1/3} \varphi(1/K(T))^5} \,.
\]
Recalling also the definition~\eqref{defpsi} of \(\psi(T) = \varphi(1/K(T))^3 S(T)\), the term appearing in~\eqref{1stand2nd-bis} is thus proportional to
\begin{equation*}
\frac{ TK(T)^{2/3}}{ S(T)^2\varphi(1/K(T))^5}= \frac{TK(T)^{2/3} \varphi(1/K(T))}{\psi(T)^{2}}  \stackrel{T\to \infty}{\sim} \frac{c_{2/3}}{ \psi(T)^{2}} \,,
\end{equation*}
using again~\eqref{implicit} for the last asymptotic.
Since $\psi$ diverges, this concludes the proof of~\eqref{1stand2nd-bis}.
\qed

\subsection{Proof of \texorpdfstring{\Cref{lem:var1}}{the lemma}}

From the Poisson construction of Section~\ref{sec:propertiesRW}, we have  
\begin{equation}
  \label{exactexpress}
 \Var_{\bbP}[F_T]= \rho T \times \sum_{k=1}^{2 \beta_0 K(T)^{-1}} \xi(k)^2 \bar \mu(k) \,.
\end{equation}
Now note that if $f$ is a positive regularly varying function, recalling that \(\bar \mu(k) = (2k+1) (J(k)-J(k+1))\) and decomposing over diadic scales we obtain
\[
\sum_{k=1}^{2^n} f(k) \bar \mu(k) \asymp \sum_{i=1}^{n}  f(2^i) 2^i \sum_{k=2^{i-1}+1}^{2^i} (J(k)-J(k+1))  
\asymp \sum_{i=1}^{2^n}  f(2^i) 2^i J(2^i)  \asymp \sum_{k=1}^{2^n} f(k) J(k) \,,
\]
where we have used the notation \(a_n\asymp b_n\) to say that there is a constant \(c>0\) such that \(c^{-1}\leq a_n/b_n \leq c\) for all \(n\geq 1\).
It is then  not difficult to check then that this also holds along the whole sequence of integers rather than just powers of $2$. Looking at the case $f(k)= \xi(k)^2 = k^{2/3} \gp(k)^{-2}$ and replacing sums by integrals, we obtain that
\begin{equation*}
  \Var_{\bbP}[F_T]\asymp  \rho T\int^{2 \beta_0 K(T)^{-1}}_1 \frac{\dd s}{s \varphi(s)^3}.
\end{equation*}
Replacing $2 \beta_0 K(T)^{-1}$ by $K(T)^{-1}$ does not alter the order of magnitude of the r.h.s.\ and concludes the proof of \eqref{frwrt}.

Now let us compute the second moment of $F_1$ under \(\bbP_{\tau}\).
The \(H_j\) in \eqref{def:F1F2} are independent under \(\bbP_{\tau}\) so that \(\Var_{\bbP_{\tau}}(F_1) =\sum_{j=1}^m \Var_{\bbP_{\tau}}(H_j)\). 
Bounding the variance by the second moment and using  stochastic comparison of \Cref{lem:domination} (note that \(H_j\) hence \(H_j\) is an non-decreasing function of \(\bar\cU\) since \(\xi\) is non-negative), we get
\[
 \Var_{\bbP_{\tau}}\left( H_j\right) \le \bbE_{\tau}\left[ (H_j)^2 \right] \leq \bbE[(H_j)^2] =\Var_{\bbP}(H_j)+ \bbE[H_j]^2 \,.
\]
Now, by definition~\eqref{def:F1F2} of \(H_j\), using regular variation as above, we obtain that
\begin{equation}
  \label{eq:EHj}
\bbE[H_j]   = \rho \, \Delta \tau_j \sum_{k= \beta_0 K(\Delta \tau_j)^{-1}}^{2 \beta_0 K(\Delta \tau_j)^{-1}} \xi(k) \bar \mu(k)  \asymp \rho \Delta \tau_j  \frac{1}{K(\Delta \tau_j)}  \xi\Big( \frac{1}{K(\Delta \tau_j)} \Big)  J\Big( \frac{1}{K(\Delta \tau_j)} \Big)  \asymp   \rho\, \xi\Big( \frac{1}{K(\Delta \tau_j)} \Big) \,,
\end{equation}
where for the last identity we have used~\eqref{implicit} to get that \(t K(t)^{-1} J(K(t)^{-1}) \asymp 1\).
Repeating the same computation we obtain that 
\[
\bbE[(H_j)^2]  \asymp \rho\, \xi\Big( \frac{1}{K(\Delta \tau_j)} \Big)^2 \,.
\]
As a consequence, if \(\rho\) is sufficiently small we have \(\bbE[H_j]^2 \leq \frac{1}{2} \bbE[(H_j)^2] \), from which we deduce that \(\bbE[H_j]^2 \leq \Var_{\bbP}(H_j)\) and thus $\Var_{\bbP_{\tau}}( H_j)\le 2\Var_{\bbP}(H_j)$.
Summing over \(j\), we finally obtain that \(\Var_{\bbP_{\tau}}(F_1) \leq 2 \Var_{\bbP}(F_1) \leq 2 \Var_{\bbP}(F_T)\), which concludes the proof.
\qed

\subsection{Proof of \texorpdfstring{\Cref{lem:shiftE}}{the lemma}}

First of all, using Mecke's formula \cite[Theorem 4.1]{LP18}, we get that  
\[
\bbE\bigg[ \sum_{j: \vartheta_j \in (\tau_{j-1},\tau_j]} \ind_{\{U_j=k\}} \bigg]  =\rho \bar \mu(k) \Delta \tau_j \,.
\]
On the other hand, using \Cref{lem:sizebiased}, we also have
\[
\bbE_{\tau}\bigg[ \sum_{i: \vartheta \in (\tau_{j-1},\tau_j]} \ind_{\{U_i=k\}} \bigg]  = \rho  \bar \mu(k) \Delta \tau_j \; \frac{1}{2k+1} \sum_{x=-k}^k \frac{\P(W_{\Delta \tau_j}=x)}{\P(W_{\Delta \tau_j}=0)} \leq  \frac12 \, \rho  \bar \mu(k) \Delta \tau_j \,,
\]
where the last inequality holds for \(k\geq \beta_0 K(\Delta \tau_j)^{-1}\) since then we have that \((2k+1) \P(W_{\Delta \tau_j}=0) = (2k+1) \beta_0^{-1} K(\Delta \tau_j)\geq 2\).
All together, we have that 
\begin{equation*}
 (\bbE- \bbE_{\tau})\big[ H_j \big] \ge  \frac12 \rho  \Delta \tau_j  \sum_{k=\beta_0 K(\Delta \tau_j)^{-1}}^{2\beta_0 K(\Delta \tau_j)^{-1}} \xi(k) \bar \mu(k) \geq c  \rho\, \xi\Big( \frac{1}{K(\Delta \tau_j)} \Big)\,,
\end{equation*}
the last inequality following as in~\eqref{eq:EHj} above.
Summing over~\(j\), this concludes the proof of Lemma~\ref{lem:shiftE}.
\qed

\noindent {\bf Acknowledgments:}
H.L.\ acknowledges the support of a productivity grand from CNQq and of a CNE grant from FAPERj.
Q.B. acknowledges the support of Institut Universitaire de France and ANR Local (ANR-22-CE40-0012-02).

\appendix

\section{Some results on the homogeneous pinning model}
\label{app:hompinning}

Recall that we have defined
\(K_\gb(t) := \frac{\gb}{\gb_0} K(t) e^{-\tf(\gb) t}\), so that $\int_0^{\infty} K_{\gb} (t) \dd t =1$ if $\gb\geq \gb_0$, and that \(\bQ_{\gb}\) denotes the law of a renewal process with inter-arrival density \(K_{\gb}\). 
We also denote \(\bar K_{\beta}(t) = \int_{t}^{\infty} K_{\beta}(s) \dd s\).

\subsection{Removing the endpoint conditioning}

Recall that we have defined the conditioned law $\bQ_{\beta,t}(  \cdot )=\lim_{\gep\to 0 }\bQ_{\beta}(  \cdot \mid \tau\cap[t,t+\gep]\ne \emptyset )$. 
We then have the following lemma, analogous to \cite[Lem.~A.2]{GLT11}. Note that we need to deal with \(\bQ_{\beta}\) for \(\beta \geq \beta_0\) instead of only \(\bQ=\bQ_{\beta_0}\), so we give a short proof of it for completeness.

\begin{lemma}
\label{lem:removecond}
Assume that \(K(t) = L(t) t^{-(1+\alpha)}\) for some slowly varying function \(L(\cdot)\) and \(\alpha>0\).
Then, there exists a constant $C$ such that, for any \(\beta\in [\beta_0,2\beta_0]\), for any $t>1$ and any non-negative measurable function~$f$
\[
\bQ_{\beta,2t}\big[ f\big( \tau\cap [0,t] \big) \big] \leq C \bQ_{\beta}\big[ f\big( \tau\cap [0,t] \big)  \big] \,.
\]
(Recall that \(\bQ=\bQ_{\beta_0}\).)
\end{lemma}

\begin{proof}
  Let \(u_{\beta}(t)\) be the density of the renewal measure, defined by on $(0,\infty)$ by \(\bQ_{\beta}[|\tau \cap A|] = \int_A u_{\beta}(t) \dd t\). With some abuse of notation we let $u_{\beta}(\dd t):= u_{\beta}(t)\dd t+ \delta_0(\dd t)$ be the associated measure (including a Dirac mass at $0$).
  Decomposing over the position \(r = \sup \tau\cap [0,t] \) and \(v =2t- \inf \tau \cap (t,2t]\), we have 
  \begin{equation*}\begin{split}
    \bQ_{\beta,2t}\big[ f\big( \tau\cap [0,t] \big) \big] 
    &= \int_0^t \frac{1}{u_{\beta}(2t)}\bQ_{\beta,r}\big[ f\big( \tau\cap [0,r] \big) \big]  \bigg( \int_{0}^{t} K_{\beta}(2t-r-v) u_{\beta}(\dd v) \bigg) u_{\beta}(\dd r) \,,\\
      \bQ_{\beta}\big[ f\big( \tau\cap [0,t] \big) \big]& = \int_0^t  \bQ_{\beta,r}\big[ f\big( \tau\cap [0,r] \big) \big]   \bar K_{\beta}(t-r) u_{\beta}(\dd r) \,.
  \end{split}\end{equation*}
We get the desired conclusion if we can show that the ratio of the integrand over $r$ is bounded that is
  \begin{equation}
    \label{eq:boundKcond}
    \sup_{r\in [0,t]} \frac{\int_{t}^{2t} K_{\beta}(s-r) u_{\beta}(2t-s) \dd s + K_{\beta}(2t-r)}{u_{\beta}(2t) \int_{t}^{\infty}K_{\beta}(s-r) \dd s }  \leq C <+\infty ,
  \end{equation}
  (in the above $s=2t-v$).
  For~\eqref{eq:boundKcond}, let us set \(t_{\beta}:=\frac12 ( t\wedge \frac{1}{\tf(\gb)})\) and split the integral on the numerator at \(2t-t_{\beta}\).
  First, note that thanks to \cite[Lemma 3.1]{BLirrel} we have that \(u_{\beta}(2t-s) \leq C u_{\beta}(2t)\) whenever \(2t-s \geq t_{\beta}\), (we use the fact that \(u(\cdot)\) is regularly varying to treat the case $\beta=\beta_0$).
  The first part of the integral is thus 
  \[
  \int_{t}^{2t-t_{\beta}} K_{\beta}(s-r)  \frac{u_{\beta}(2t-s)}{u_{\beta}(2t)}\dd s \leq C \int_{t}^{2t-t_{\beta}} K_{\beta}(s-r)  \dd s  \,.
  \]
  For the second part of the integral, we have \(K_{\beta}(s-r) \leq c  e^{-(2t-r) \tf(\gb)} K(t) \) uniformly for  for \(s \in [2t-t_\beta,t]\), using that \(2t-r \ge \frac12 t\) and \(t_{\gb} \tf(\gb) \leq \frac12\).
  Now, using again~\cite[Lemma 3.1]{BLirrel} (in the first and last inequality) and regular variation (in the middle one) we have  $$\int_0^{t_{\beta}} u_{\beta}(x) \dd x \leq c \int_0^{t_{\beta}} u(x) \dd x \le c' t_{\gb} u(t_{\beta})\le c'' t_{\beta} u_{\beta}(t_\beta)$$ 
  Altogether, we have that 
  \[
  \int_{2t-t_{\beta}}^{2t} K_{\beta}(s-r)  \frac{u_{\beta}(2t-s)}{u_{\beta}(2t)}\dd s \leq c\,  t_{\beta} e^{-(2t-r) \tf(\gb)} K(t) \leq C \int_{2t-t_{\beta}}^{2t} K_{\beta}(s-r) \dd s \,.
  \]
  It only remains the last term. Note that we have that 
  \[
  \int_{2t}^{\infty} K_{\beta}(2s-r) \dd s \geq e^{-(2t-r)\tf(\gb)} \int_{2t}^{2t +1/\tf(\gb)} K(2s-r) \dd s \geq c \tf(\gb)^{-1} K_{\gb}(2t-r) \,.
  \] 
  Hence, it only remains to see that we have \(u_{\gb}(t) \geq c u(t\wedge \tf(\gb)^{-1}) \geq c' \tf(\gb)\) by combining ~\cite[Lemma~3.1]{BLirrel} and the fact that \(u(s) \geq c (1+s)^{-1}\) (recall~\eqref{renewalr}-\eqref{renewal}).
  This concludes the proof of \eqref{eq:boundKcond}.
\end{proof}

\subsection{About the mean, truncated mean and Laplace transform}

We focus on this section on the case \(\alpha \in (0,1)\) for simplicity (we use the results in the case \(\gamma \in (\frac23,1)\), \textit{i.e.}\ \(\alpha \in (0,\frac12)\)).
The case \(\alpha>1\) is only simpler.
Let us set 
\[
 m_{\gb} := \bQ_{\beta}[\tau_1] = \int_0^{\infty} s K_{\gb}(s) \dd s \,.
\]
and note that since \(tK(t)\) is regularly varying with exponent \(-\alpha\), we get that 
\(
m(t) \asymp t^2 K(t)\,,
\)

\begin{lemma}
  \label{lem:mbeta}
Assume that \(\alpha\in (0,1)\).
We have the following comparison: there are constants \(c,c'\) such that, for $\beta\in (\beta_0,2\beta_0)$
  \[
    c' \tf(\beta)^{-2}K\big(1/\tf(\beta)\big)  \leq \bQ_{\beta}[\tau_1 \ind_{\{\tau_1 \leq 1/\tf(\beta)\}}] \leq m_{\gb} = \bQ_{\beta}[\tau_1] \leq c \tf(\beta)^{-2}K\big(1/\tf(\beta) \big)\,.
  \]
In particular, \(\bQ_{\beta}[\tau_1 \ind_{\{\tau_1 \leq 1/\tf(\beta)\}}] \geq c' c^{-1} m_{\beta}\).
\end{lemma}

\begin{proof}
First of all, note that
\[
  m_{\gb} =\frac{\gb}{\gb_0}   \int_0^{1/\tf(\gb)} e^{-s \tf(\beta)} s K(s) \dd s +\frac{\gb}{\gb_0} \int^{\infty}_{1/\tf(\gb)} e^{-s \tf(\beta)}t K(t) \dd s  \,.
\]
The first term is \(\bQ_{\beta}[\tau_1 \ind_{\{\tau_1 \leq 1/\tf(\beta)\}}]\) and is comparable with \(\int_0^{1/\tf(\gb)} sK(s) \dd s\).
Now, by regular variation of \(sK(s)\) we get that \(\int_0^{t} sK(s) \dd s \sim t^2 K(t)\) as \(t\to\infty\), so the first term is comparable to \(\tf(\beta)^{-2}K (1/\tf(\beta))\).
For the second one, we have 
\begin{equation*}
  \int^{\infty}_{1/\tf(\gb)} s K(s)e^{- s \tf(\beta)} \dd t= \Big(\sup_{s \geq 1/\tf(\beta)} s K(s)\Big)   \tf(\beta)^{-1}
  \le c   \tf(\beta)^{-2} K\big(1/\tf(\gb)\big) \,,
\end{equation*}
since regular variation implies that $\sup_{s\ge t} s K(s)\stackrel{t\to \infty}{\sim} tK(t)$. This completes the proof.
\end{proof}

\begin{lemma}
  \label{lem:Laplace}
  There is a constant $C>0$ such that, for any $\lambda \in (0, \frac12 \tf(\gb))$ we have
  \begin{equation}
  \bQ_{\gb}\big[ e^{\lambda \tau_1}\big] \leq 1+\lambda m_{\beta}+ C\lambda^2 m_{\beta} \tf(\beta)^{-1} \,.
  \end{equation}
\end{lemma}

\begin{proof}
  For any $\lambda\in (0, \frac12 \tf(\gb))$, we have
  \begin{align*}
    \partial^2_{\gl} \bQ_{\gb}\big[ e^{\lambda \tau_1}\big] = \bQ_{\gb}\big[ \tau^2_1 e^{\lambda \tau_1}\big] 
    =  \frac{\beta}{\beta_0} \int_0^{+\infty}  e^{-s (\tf(\beta)-\lambda) } s^2K(s) \dd s
    \le \frac{\beta}{\beta_0} \int_0^{+\infty}  e^{- \frac12s\tf(\beta) } s^2 K(2) \dd s  \,.
  \end{align*}
  We proceed then as in the proof of \Cref{lem:mbeta},splitting the integral according to $s\le \tf(\beta)^{-1}$  and $s> \tf(\beta)^{-1}$, it yields that
  \[
  \int_0^{+\infty}  e^{- \frac12s\tf(\beta) } s^2 K(2) \dd s \leq C \, \tf(\beta)^{-3} K(1/\tf(\beta))\le C' m_{\beta}/\tf(\beta)\,,
  \]
  using \Cref{lem:mbeta} for the last inequality.
  We then deduce the result from Taylor's formula.
\end{proof}

\bibliographystyle{abbrv}
\bibliography{biblio.bib}

\end{document}